\pdfoutput=1
\documentclass[a4paper,12pt,reqno,oneside]{amsart}

\usepackage[utf8x]{inputenc}
\usepackage[T1]{fontenc}
\usepackage{geometry}
\usepackage{lmodern}
\usepackage[stretch=10]{microtype}

\usepackage{amsmath}
\usepackage{amssymb}
\usepackage{amsthm}
\usepackage[matrix,arrow,cmtip]{xy}
\usepackage{enumerate}
\PassOptionsToPackage{hyphens}{url}
\usepackage[bookmarksnumbered,colorlinks,linkcolor=black,citecolor=black,urlcolor=black]{hyperref}
\usepackage[capitalise]{cleveref}

\newcommand{\bC}{\mathbb{C}}

\newcommand{\bG}{\mathbb{G}}
\newcommand{\bH}{\mathbb{H}}

\newcommand{\bP}{\mathbb{P}}
\newcommand{\bQ}{\mathbb{Q}}
\newcommand{\bR}{\mathbb{R}}
\newcommand{\bS}{\mathbb{S}}
\newcommand{\bZ}{\mathbb{Z}}
\newcommand{\ZZ}{\mathbb{Z}}
\newcommand{\QQ}{\mathbb{Q}}
\newcommand{\RR}{\mathbb{R}}
\newcommand{\CC}{\mathbb{C}}

\newcommand{\GG}{\mathbb{G}}

\newcommand{\cA}{\mathcal{A}}
\newcommand{\cF}{\mathcal{F}}
\newcommand{\cH}{\mathcal{H}}

\newcommand{\fp}{\mathfrak{p}}

\newcommand{\fS}{\mathfrak{S}}

\newcommand{\gG}{\mathbf{G}}
\newcommand{\gGL}{\mathbf{GL}}
\newcommand{\gGSp}{\mathbf{GSp}}
\newcommand{\gH}{\mathbf{H}}

\newcommand{\gM}{\mathbf{M}}

\newcommand{\gO}{\mathbf{O}}
\newcommand{\gP}{\mathbf{P}}

\newcommand{\gQ}{\mathbf{Q}}
\newcommand{\gS}{\mathbf{S}}
\newcommand{\gSL}{\mathbf{SL}}
\newcommand{\gSO}{\mathbf{SO}}
\newcommand{\gSp}{\mathbf{Sp}}
\newcommand{\gT}{\mathbf{T}}
\newcommand{\gU}{\mathbf{U}}
\newcommand{\gZ}{\mathbf{Z}}

\newcommand{\rH}{\mathrm{H}}
\newcommand{\rM}{\mathrm{M}}

\newcommand{\rO}{\mathrm{O}}

\DeclareMathOperator{\Aut}{Aut}

\DeclareMathOperator{\diag}{diag}
\DeclareMathOperator{\disc}{disc}
\DeclareMathOperator{\End}{End}

\DeclareMathOperator{\GL}{GL}
\DeclareMathOperator{\Gr}{Gr}

\DeclareMathOperator{\im}{im}

\DeclareMathOperator{\Lie}{Lie}

\DeclareMathOperator{\Nm}{Nm}
\DeclareMathOperator{\Nrd}{Nrd}
\DeclareMathOperator{\OGr}{OGr}
\DeclareMathOperator{\Res}{Res}
\DeclareMathOperator{\rk}{rk}
\DeclareMathOperator{\Sh}{Sh}
\DeclareMathOperator{\SL}{SL}
\DeclareMathOperator{\SO}{SO}

\DeclareMathOperator{\Stab}{Stab}
\DeclareMathOperator{\tr}{tr}
\DeclareMathOperator{\Tr}{Tr}
\DeclareMathOperator{\Trd}{Trd}

\newcommand{\der}{\mathrm{der}}

\newcommand{\op}{\mathrm{op}}

\newcommand{\Quat}{\mathrm{Quat}}

\newcommand{\extpower}{\bigwedge\nolimits}
\newcommand{\extpowerdec}{\bigwedge\nolimits_{\mathrm{dec}}}

\newcommand{\abs}[1]{\lvert #1 \rvert}
\newcommand{\length}[1]{\lVert #1 \rVert}

\newcommand{\bs}{\backslash}
\newcommand{\eps}{\varepsilon}
\newcommand{\ov}{\overline}
\newcommand{\half}{\tfrac{1}{2}}

\newcommand{\fullmatrix}[4]{\left( \begin{matrix} #1 & #2 \\ #3 & #4 \end{matrix} \right)}
\newcommand{\fullsmallmatrix}[4]{\bigl( \begin{smallmatrix} #1 & #2 \\ #3 & #4 \end{smallmatrix} \bigr)}

\newcommand{\defterm}[1]{\textbf{#1}}

\newtheorem{lemma}{Lemma}[section]
\newtheorem{proposition}[lemma]{Proposition}
\newtheorem{theorem}[lemma]{Theorem}
\newtheorem{corollary}[lemma]{Corollary}
\newtheorem{conjecture}[lemma]{Conjecture}
\newtheorem{properties}[lemma]{Properties}
\Crefname{conjecture}{Conjecture}{Conjectures} 

\Crefname{claim}{Claim}{Claims}

\newtheorem*{lemma*}{Lemma}
\newtheorem*{proposition*}{Proposition}
\newtheorem*{theorem*}{Theorem}
\newtheorem*{corollary*}{Corollary}
\newtheorem*{claim*}{Claim}

\theoremstyle{definition}
\newtheorem*{definition}{Definition}
\newtheorem{remark}[lemma]{Remark}

\usepackage{ifthen}
\newcounter{constant}
\newcommand{\newC}[1]{%
   \refstepcounter{constant} c_{\theconstant}%
   \ifthenelse{\equal{#1}{*}} { } {%
      \label{c:#1}%
   }%
}
\newcommand{\refC}[1]{c_{\ref*{c:#1}}}

\newcommand{\Qbar}{\overline \bQ}
\newcommand{\AAA}{\mathbb{A}}
\newcommand{\fA}{\mathfrak{A}}

\newcommand{\ExCM}{\( E \, \times \! \) CM}

\title{Quantitative reduction theory and unlikely intersections}

\author{Christopher Daw}

\author{Martin Orr}

\address{Daw: Department of Mathematics and Statistics, University of Reading,
    White\-knights,  PO Box 217,  Reading,  Berkshire RG6 6AH,  United Kingdom}
\email{chris.daw@reading.ac.uk}

\address{Orr: Department of Mathematics, The University of Manchester, Alan Turing Building, Oxford Road, Manchester M13 9PL, United Kingdom}
\email{martin.orr@manchester.ac.uk}


\begin{document}

\begin{abstract}
We prove quantitative versions of Borel and Harish-Chandra's theorems on reduction theory for arithmetic groups. Firstly, we obtain polynomial bounds on the lengths of reduced integral vectors in any rational representation of a reductive group. Secondly, we obtain polynomial bounds in the construction of fundamental sets for arithmetic subgroups of reductive groups, as the latter vary in a real conjugacy class of subgroups of a fixed reductive group.

Our results allow us to apply the Pila--Zannier strategy to the Zilber--Pink conjecture for the moduli space of principally polarised abelian surfaces. Building on our previous paper, we prove this conjecture under a Galois orbits hypothesis. Finally, we establish the Galois orbits hypothesis for points corresponding to abelian surfaces with quaternionic multiplication, under certain geometric conditions.
\end{abstract}

\maketitle

This is a pre-copyedited, author-produced version of an article accepted for publication in International Mathematics Research Notices following peer review. The version of record (published online 16 July 2021) is available online under the CC BY licence at:
\url{https://doi.org/10.1093/imrn/rnab173}

\section{Introduction}

Reduction theory is concerned with finding small representatives for each orbit in actions of arithmetic groups, for example through constructing fundamental sets.
It began with the study of the action of \( \SL_n(\bZ) \) on quadratic forms, which was described by Siegel in terms of a fundamental set for \( \SL_n(\bZ) \) in \( \SL_n(\bR) \).
Borel and Harish-Chandra generalised this to arithmetic lattices in arbitrary semisimple Lie groups.
This theory has had wide-ranging applications in areas such as the theory of automorphic forms and locally symmetric spaces \cite{AMRT}, the arithmetic of algebraic groups \cite{PR94} and finiteness theorems for abelian varieties~\cite{Mil86}.

The first goal of this paper is to prove quantitative bounds for the group elements used in Borel and Harish-Chandra's construction of fundamental sets.
These bounds are polynomial in terms of suitable input parameters, although they are not fully effective.
They generalise the polynomial bounds of Li and Margulis for the reduction theory of quadratic forms \cite{LM16} and complement the second-named author's polynomial bounds for the Siegel property \cite{Orr18}. 
It should be noted that while Borel and Harish-Chandra's reduction theory is algorithmic in nature, as made explicit by Grunewald and Segal \cite{GS80}, their arguments give no bounds for the running time or output size of these algorithms.

Our primary theorems on reduction theory are as follows. The first is a quantitative version of \cite[Lemma~5.4]{BHC62}.  See section~\ref{sec:preliminaries} for the relevant definitions and section~\ref{subsec:BHC62i} for discussion of how this theorem is related to \cite{BHC62}.

\begin{theorem} \label{siegel-intersection}
Let \( \gG \) be a reductive \( \bQ \)-algebraic group and let \( \fS \subset \gG(\bR) \) be a Siegel set.
Let \( \rho \colon \gG \to \gGL(V) \) be a representation of \( \gG \) defined over \( \bQ \).
Let \( \Lambda \subset V \) be a \( \bZ \)-lattice.
Let \( v_0 \in V_\bR \) be such that:
\begin{enumerate}[(i)]
\item \( \rho(\gG(\bR))v_0 \) is closed in \( V_\bR \);
\item the stabiliser \( \Stab_{\gG(\bR),\rho} (v_0) \) is self-adjoint.
\end{enumerate}
Then there exist constants \( \newC{siegel-intersection-multiplier} \), \( \newC{siegel-intersection-exponent} \) such that,
for every \( v \in \Aut_{\rho(\gG)}(V_\bR)v_0 \) and every \( w \in \rho(\fS) v \cap \Lambda \), we have \( \abs{w} \leq \refC{siegel-intersection-multiplier} \abs{v}^{\refC{siegel-intersection-exponent}} \).
\end{theorem}

We use Theorem \ref{siegel-intersection} to prove our second theorem on quantitative reduction theory, which is a quantitative version of Borel and Harish-Chandra's construction of fundamental sets for arithmetic groups \cite[Thm.~6.5]{BHC62}.  See section~\ref{subsec:BHC62ii} for discussion of the details of this theorem, including how it relates to \cite{BHC62}.

\begin{theorem} \label{fund-set-bound}
Let \( \gG \) be a reductive \( \bQ \)-algebraic group.
Let \( \Gamma \subset \gG(\bQ) \) be an arithmetic subgroup.
Let \( \fS \subset \gG(\bR) \) be a Siegel set such that \( C\fS \) is a fundamental set for \( \Gamma \) in \( \gG(\bR) \), for some finite set \( C \subset \gG(\bQ) \).

Let \( \rho \colon \gG \to \gGL(\Lambda_\bQ) \) be a \( \bQ \)-algebraic representation of \( \gG \), where \( \Lambda \) is a finitely generated free \( \bZ \)-module.
Let \( \gH_0 \subset \gG \) be a self-adjoint reductive \( \bQ \)-algebraic subgroup and let \( v_0 \in \Lambda \) be a vector such that:
\begin{enumerate}[(i)]
\item \( \Stab_{\gG,\rho}(v_0) = \gH_0 \);
\item the orbit \( \rho(\gG(\bR)) v_0 \) is closed in \( \Lambda_\bR \).
\end{enumerate}

Then there exist positive constants \( \newC{fund-bound-multiplier} \) and \( \newC{fund-bound-exponent} \) (depending only on \( \gG \), $\Gamma$, $\fS$, $C$, \( \rho \), \( \gH_0 \), and~\( v_0 \)) with the following property:
for every \( u \in \gG(\bR) \) and \( v_u \in \Aut_{\rho(\gG)}(\Lambda_\bR) v_0 \) such that \( \gH_u = u \gH_{0,\bR} u^{-1} \) is defined over~\( \bQ \) and \( \rho(u)v_u \in \Lambda \), there exists a fundamental set for \( \Gamma \cap \gH_u(\bR) \) in \( \gH_u(\bR) \) of the form
\[ B_u C \fS u^{-1} \cap \gH_u(\bR), \]
where \( B_u \subset \Gamma \) is a finite set such that every \( b \in B_u \) satisfies
\[ \abs{\rho(b^{-1}u) v_u}  \leq  \refC{fund-bound-multiplier} \abs{v_u}^{\refC{fund-bound-exponent}}. \]
\end{theorem}

We apply Theorem \ref{fund-set-bound} to the Zilber--Pink conjecture on unlikely intersections.
We prove the Zilber--Pink conjecture for \( \cA_2 \), the (coarse) moduli space of principally polarised abelian surfaces over~$\CC$, subject to a large Galois orbits conjecture (Conjecture \ref{GOSh}), which is stated in section~\ref{sec:unlikely-int-proofs}.

We recall that $\mathcal{A}_2$ is a Shimura variety of dimension~$3$ associated with the group $\gGSp_4$.
The Zilber--Pink conjecture predicts that an irreducible algebraic curve $C\subset\cA_2$ that is Hodge generic (that is, not contained in any proper special subvariety) contains only finitely many points of intersection with the special subvarieties of $\mathcal{A}_2$ having dimension $1$ or $0$ (see \cite[Conjecture 1.3]{pink:generalisation} for the most relevant formulation of the Zilber--Pink conjecture).

Pila and Tsimerman's proof of the Andr\'e--Oort conjecture for $\cA_2$ \cite{PT13} shows that $C$ contains only finitely many special points. (A special point on $\cA_2$ is a point associated with an abelian surface with complex multiplication.) Therefore, in order to prove the Zilber--Pink conjecture for $\cA_2$, it suffices to show that $C$ contains only finitely many non-special points belonging also to a special curve.

The special curves in \( \cA_2 \) are of three types:
\begin{enumerate}
\item curves parametrising abelian surfaces with quaternionic multiplication (we refer to these as quaternionic curves);
\item curves parametrising abelian surfaces isogenous to the square of an elliptic curve (``$E^2$ curves'');
\item curves parametrising abelian surfaces isogenous to the product of two elliptic curves, at least one of which has complex multiplication (``\ExCM{} curves'').
\end{enumerate}
In this paper, we study intersections with the quaternionic and $E^2$ curves. Let $\Sigma_{\Quat}$ (resp. $\Sigma_{E^2}$) denote the set of points of $\mathcal{A}_2$ which are Hodge generic in some quaternionic (resp. $E^2$) curve. Our first main result on unlikely intersections is the following.

\begin{theorem}\label{ZPSh}
Let $\Sigma$ denote $\Sigma_{\Quat}$ or $\Sigma_{E^2}$ and let $C\subset\mathcal{A}_2$ denote an irreducible Hodge generic algebraic curve. 

If $C$ satisfies Conjecture \ref{GOSh} for $\Sigma$, then $C\cap\Sigma$ is finite.
\end{theorem}

Combined with our previous work \cite{DO19}, in which we study intersections with the \ExCM{} curves, and Pila and Tsimerman's proof of the André--Oort conjecture for~\( \cA_2 \), this completes the proof of the Zilber--Pink conjecture for \( \cA_2 \), subject to Conjecture~\ref{GOSh} (a large Galois orbits conjecture). As in \cite{DO19}, the general strategy follows the proof of \cite[Theorem 14.2]{DR}, which was an application of the so-called Pila--Zannier method to the Zilber--Pink conjecture for general Shimura varieties. However, we will have to make several modifications, and the end of the proof is closer to \cite[Proposition~3.5]{OrrUI}.

Finally, we show that the large Galois orbits conjecture holds for $\Sigma=\Sigma_{\Quat}$ when the curve under consideration satisfies a multiplicative reduction hypothesis at the boundary of the moduli space. We proved the analogous result for intersections with \ExCM{} curves in \cite{DO19}.

\begin{theorem}\label{ZPShUn}
Let $C\subset\mathcal{A}_2$ denote an irreducible Hodge generic algebraic curve defined over \( \Qbar \) such that the Zariski closure of \( C \) in the Baily--Borel compactification of \( \cA_2 \) intersects the \( 0 \)-dimensional stratum of the boundary.

Then $C$ satisfies Conjecture \ref{GOSh} for $\Sigma_{\Quat}$, and so \( C \cap \Sigma_{\Quat} \) is finite.
\end{theorem}

\cref{ZPShUn} follows from a result of André \cite[Ch.~X, Thm.~1.3]{And89} and the Masser--Wüstholz isogeny theorem.
We have not been able to prove an analogue of \cref{ZPShUn} for $\Sigma=\Sigma_{E^2}$ because the result of André does not apply to abelian surfaces isogenous to the square of an elliptic curve.

We note that the results on quantitative reduction theory in this paper will be an important tool for proving the Zilber--Pink conjecture for other Shimura varieties, which will be the subject of future work by the authors.
We expect these results to have further applications, for example, a uniform version of the second-named author's bounds for polarisations and isogenies of abelian varieties \cite{Orr17} and bounds for the heights of generators of arithmetic groups by combining them with some of the techniques of homogeneous dynamics from \cite{LM16}.

\subsection{Outline of the paper}

In section~\ref{sec:background}, we give some background on reduction theory to put our \cref{siegel-intersection,fund-set-bound} into context. 

In section~\ref{sec:preliminaries}, we define notation to be used throughout the paper. We state the various (equivalent) definitions of a Cartan involution that exist in the literature, and we define the notion of a Siegel set.
 
In section~\ref{sec:qrt}, we prove our main theorems on quantitative reduction theory, namely, \cref{siegel-intersection,fund-set-bound}. The proof follows the stategy of \cite{BHC62} but, in order to obtain a quantitative result, it is necessary to replace ``soft'' topological ingredients in the former, notably \cite[Prop. 5.2]{BHC62}, with results from elsewhere.

In order to apply \cref{fund-set-bound} to a specific situation, we must construct a representation $\rho$ of the ambient reductive group and show that the vectors $v_u$ can be chosen suitably bounded.
This is the topic of section~\ref{sec:effective-algebras} -- we construct a representation of \( \gGSp_4 \) with the properties required to apply \cref{fund-set-bound} to the subgroups associated with quaternionic and \( E^2 \) curves.

In section~\ref{sec:unlikely-int-proofs}, we prove our theorems on unlikely intersections, namely, \cref{ZPSh,ZPShUn}. 
The strategy for the former is to parametrise the unlikely intersections by integral vectors of suitably bounded length (or, equivalently, height). This parametrisation is obtained using the results of sections \ref{sec:qrt} and~\ref{sec:effective-algebras}. Then, as in all versions of the Pila--Zannier method, we consider the parameters for unlikely intersections as a set definable in an o-minimal structure and apply a Pila--Wilkie counting theorem to control the number of such points in terms of their height.  \Cref{ZPSh} follows by comparing this upper bound with the lower bound of the large Galois orbits conjecture. 

We emphasise that sections~\ref{sec:background}--\ref{sec:effective-algebras} require no knowledge of Shimura varieties and may be read independently of section~\ref{sec:unlikely-int-proofs}, while the results of section~\ref{sec:background} are not used elsewhere in the paper.

\section{Background on reduction theory}\label{sec:background}

In this section we outline some of the history of reduction theory, focussing on quantitative results and on the work of Borel and Harish-Chandra, and explain how our \cref{siegel-intersection,fund-set-bound} fit into this theory.
The results of this section are not used later in the paper.

\subsection{Reduction theory for quadratic forms}

The group \( \gSL_n(\bZ) \) acts on the set of integral quadratic forms in \( n \) variables via its natural action on the variables.
The classical reduction theory of quadratic forms defines a set of \emph{reduced quadratic forms} with the following properties.

\begin{properties} \label{prop:reduced} \leavevmode
\begin{enumerate}[(i)]
\item Each \( \gSL_n(\bZ) \)-orbit of non-degenerate integral quadratic forms contains at least one reduced form.
\item Each \( \gSL_n(\bZ) \)-orbit of non-degenerate integral quadratic forms contains only finitely many reduced forms.
\end{enumerate}
\end{properties}

A variety of definitions of reduced quadratic forms are used, possessing varying properties in addition to \cref{prop:reduced}.
The most important definitions are due to Lagrange and Gauss for binary quadratic forms, and to Hermite, Minkowski and Siegel for quadratic forms in any number of variables.
In these definitions (except some of Hermite's definitions), the reduced quadratic forms can be defined by finitely many polynomial inequalities in the coefficients of the forms.


Reduction theory behaves better for positive (or negative) definite quadratic forms than for indefinite forms.
Definite quadratic forms satisfy a much stronger version of Property~\ref{prop:reduced}(ii) called the Siegel property.
This guarantees that there is a uniform bound on the number of reduced quadratic forms in each \( \gSL_n(\bZ) \)-orbit (for fixed~\( n \)).
In the nicest case of all, definite binary quadratic forms using Gauss's definition of reduced forms, each \( \gSL_2(\bZ) \)-orbit contains \emph{exactly one} reduced form.


\subsubsection*{Quantitative reduction theory for quadratic forms}

The discriminant of a quadratic form is invariant under the action of \( \gSL_n(\bZ) \).  A form is non-degenerate if and only if its discriminant is non-zero.
Hence the following lemma implies Property~\ref{prop:reduced}(i).

\begin{lemma} \label{disc-finite-reduced}
For each integer~\( \Delta \neq 0 \), there are only finitely many reduced integral quadratic forms in \( n \) variables of discriminant~\( \Delta \).
\end{lemma}

The following quantitative version of \cref{disc-finite-reduced} is classical for anisotropic quadratic forms \cite[p.~287, Cor.~1]{Cas78}, using Hermite and Minkowski's definitions of reduction (note that anisotropic forms are always definite when \( n \geq 5 \)).
Indeed, for binary anisotropic forms, it goes back to Lagrange.
For Siegel reduced forms (definite or indefinite), it can be proved by adapting the proof of \cref{disc-finite-reduced} found on \cite[pp.~322--324]{Cas78} but we are unsure whether this was classically known.

\begin{proposition} \label{reduced-disc-bound}
For each positive integer~\( n \), there exists a positive real number \( \newC{reduced-disc-bound}(n) \) such that, for every integer~\( \Delta \neq 0 \), all reduced integral \( n \)-ary quadratic forms of discriminant \( \Delta \) have coefficients with absolute values at most \( \refC{reduced-disc-bound}(n) \abs{\Delta} \).
\end{proposition}

\Cref{disc-finite-reduced} also implies the following lemma, which does not mention reduced forms, and which was historically one of the most important consequences of reduction theory.

\begin{lemma} \label{orbits-finite}
For each integer~\( \Delta \neq 0 \), there are only finitely many \( \gSL_n(\bZ) \)-orbits of integral quadratic forms in \( n \) variables of discriminant~\( \Delta \).
\end{lemma}

Li and Margulis have proved a quantitative version of \cref{orbits-finite}.
The bound is stronger than in \cref{reduced-disc-bound}, but it applies only to at least one form in each \( \gSL_n(\bZ) \)-orbit, rather than to all reduced forms.

\begin{proposition} \cite[Theorem~3]{LM16}
For each integer~\( n \geq 3 \), there exists a constant \( \newC{lm-indef-bound}(n) \) such that every \( \gSL_n(\bZ) \)-orbit of indefinite quadratic forms in \( n \) variables of discriminant~\( \Delta \neq 0 \) contains a form whose coefficients have absolute value at most \( \refC{lm-indef-bound}(n) \abs{\Delta}^{1/n} \).
\end{proposition}

\subsection{Siegel sets}

Siegel shifted the emphasis in reduction theory from quadratic forms to arithmetic groups.
There is a direct link between the reduction theory of definite quadratic forms and fundamental sets for \( \gSL_n(\bZ) \) in \( \gSL_n(\bR) \).

Let \( v_0 \) denote the standard positive definite quadratic form in \( n \) variables:
\[ v_0(x_1, \dotsc, x_n) = x_1^2 + \dotsb + x_n^2. \]
If \( \cF \) is a fundamental set of \emph{positive definite real} quadratic forms in \( n \) variables (that is, a set which satisfies the generalisations of \cref{prop:reduced} for positive definite real forms), then
\[ \fS = \{ g \in \gSL_n(\bR) : gv_0 \in \cF \} \]
is a fundamental set in \( \gSL_n(\bR) \) for the action of \( \gSL_n(\bZ) \) by multiplication on the left (that is, every right \( \gSL_n(\bZ) \)-coset intersects \( \fS \) in at least one, and at most finitely many, elements).
Conversely, if \( \fS \subset \gSL_n(\bR) \) is a fundamental set for \( \gSL_n(\bZ) \) which is invariant under right multiplication by \( \SO_n(\bR) = \Stab_{\gSL_n(\bR)}(v_0) \), then \( \bR_{>0} \fS v_0 \) is a fundamental set of positive definite real quadratic forms.

Siegel defined a family of sets \( \fS = \fS_{t,u} \subset \gSL_n(\bR) \), depending on two parameters \( t, u \in \bR_{>0} \).
The set \( \fS_{t,u} \) is a fundamental set for \( \gSL_n(\bZ) \) whenever \( t \leq \sqrt{3}/2 \) and \( u \geq 1/2 \) (according to the conventions used in this paper).
We call \( \fS_0 = \fS_{\sqrt{3}/2,1/2} \) the \defterm{standard Siegel set} in \( \gSL_n(\bR) \).
Using the construction described in the previous paragraph, we obtain a fundamental set of positive definite real quadratic forms, namely \( \bR_{>0} \fS_0 v_0 \).
We say that a positive definite quadratic form is \defterm{Siegel reduced} if it lies in \( \bR_{>0} \fS_0 v_0 \).

We also say that an indefinite quadratic form of signature \( (p,q) \) (with $p+q=n$) is \defterm{Siegel reduced} if it lies in \( \bR_{>0} \fS_0 v_0^{(p,q)} \), where
\begin{equation} \label{eqn:standard-quadratic-form}
v_0^{(p,q)}(x_1, \dotsc, x_{p+q}) = x_1^2 + \dotsb + x_p^2 - x_{p+1}^2 - \dotsb - x_{p+q}^2.
\end{equation}
The Siegel reduced indefinite \textit{integral} quadratic forms satisfy \cref{prop:reduced}.
However the set of Siegel reduced indefinite real quadratic forms is not a fundamental set because it does not satisfy the generalisation of Property~\ref{prop:reduced}(ii) to real forms.

Borel and Harish-Chandra generalised the notion of Siegel set from \( \gSL_n(\bR) \) to all reductive Lie groups \cite[sec.~4.1]{BHC62} (which they called \defterm{Siegel domains}).
However these Siegel domains are not always fundamental sets for arithmetic subgroups --- in general, one can only say that there is a fundamental set \emph{contained in} a finite union of translates of Siegel domains \cite[Thm.~6.5, Lemma~7.5]{BHC62}.

Borel subsequently gave a new definition of Siegel sets for reductive \( \bQ \)-algebraic groups \cite[12.3]{Bor69}, taking into account the \( \bQ \)-algebraic group structure and not just the Lie group structure.
For each reductive \( \bQ \)-algebraic group~\( \gG \) and each arithmetic subgroup \( \Gamma \subset \gG(\bQ) \), there is a finite union of \( \gG(\bQ) \)-translates of a Siegel set which forms a fundamental set for \( \Gamma \) in \( \gG(\bR) \) \cite[Thm.~13.1]{Bor69} (this is a consequence of \cref{bhc-fundamental-sets} below).
In this paper we shall use a minor modification of Borel's definition of Siegel sets, described in section~\ref{ssec:siegel-sets}.

\subsection{Reduction theory for representations of reductive groups}\label{subsec:BHC62i}

The following result is a key step in Borel and Harish-Chandra's construction of fundamental sets for arithmetic groups.

\begin{theorem} \cite[Lemma~5.4]{BHC62} \label{svl-finite}
Let \( \gG \) be a reductive \( \bQ \)-algebraic group whose \( \bQ \)-rank is equal to its \( \bR \)-rank.
Let \( \fS \subset \gG(\bR) \) be a Siegel set.
Let \( \rho \colon \gG \to \gGL(V) \) be a representation of \( \gG \) defined over \( \bQ \).
Let \( \Lambda \subset V \) be a \( \bZ \)-lattice.
Let \( v \in V_\bR \) be such that:
\begin{enumerate}[(i)]
\item the orbit \( \rho(\gG(\bR))v \) is closed in \( V_\bR \);
\item the stabiliser \( \Stab_{\gG(\bR),\rho} (v) \) is self-adjoint.
\end{enumerate}
Then \( \rho(\fS) v \cap \Lambda \) is finite.
\end{theorem}

The restriction on the \( \bQ \)-rank of~\( \gG \) in \cref{svl-finite} can be removed with only minor alterations to the proof, provided we use the definition of Siegel sets from \cite[12.3]{Bor69} (or the definition in section~\ref{ssec:siegel-sets} of this paper) instead of the definition of Siegel domains from \cite[4.1]{BHC62}.

As noted in \cite[Example~5.5]{BHC62}, \cref{svl-finite} implies \cref{disc-finite-reduced}, by applying it to the representation of \( \gSL_n(\bR) \) on the vector space of real quadratic forms in \( n \) variables, with \( v = \lambda v_0^{(p,q)} \) where \( \lambda \in \bR_{>0} \) and \( v_0^{(p,q)} \) is defined by equation~\eqref{eqn:standard-quadratic-form}.
Then the orbit \( \gSL_n(\bR)v \) is the set of all quadratic forms of signature \( (p,q) \) and discriminant \( (-1)^q \lambda^n \), and \( \fS v \cap \Lambda \) is the set of Siegel reduced integral quadratic forms of given signature and discriminant.

In general, we may think of \( \rho(\fS) v \) as a set of ``reduced vectors'' in the representation \( V_\bR \).  However we should note that this set depends on~\( v \).

\subsubsection*{Quantitative reduction theory for representations}

Theorem \ref{siegel-intersection} is a quantitative version of \cref{svl-finite}, bounding the length of ``reduced integral vectors'', that is elements of the finite set \( \rho(\fS) v \cap \Lambda \), in terms of~\( v \) (for fixed group~\( \gG \) and representation~\( \rho \)).
We are not able to prove such a bound for all \( v \in V_\bR \): we must restrict to a set of \( v \) for which \( \gG(\bR) \) acts ``in a similar way'' on all of the permitted vectors~\( v \).
This is achieved through the condition that \( v \) must lie in the \( \Aut_{\rho(\gG)}(V_\bR) \)-orbit of a fixed vector.

For an example application of Theorem \ref{siegel-intersection}, let \( \rho \) be the representation of \( \gG = \gSL_n \) on the quadratic forms in \( n \) variables, and let \( v_0 = v_0^{(p,q)} \).
Noting that every scalar \( \lambda \in \bR^\times \) is in \( \Aut_{\rho(\gG)}(V_\bR) \), we deduce that there are constants \( \newC{quad-poly-multiplier} \) and \( \newC{quad-poly-exponent} \) (depending on~\( n \)) such that
\[ \abs{w} < \refC{quad-poly-multiplier} \abs{\lambda}^{\refC{quad-poly-exponent}} \text{ for all } \lambda \in \bR^\times \text{ and } w \in \fS \lambda v_0 \cap L. \]
Thus \cref{siegel-intersection} implies a weakened version of \cref{reduced-disc-bound} -- \cref{reduced-disc-bound} is stronger because it gives a bound which is linear in the discriminant, while the constants in \cref{siegel-intersection}, even the exponent, are ineffective (see \cref{ineffectivity}).

In this example, the closed orbits in \( V_\bR \) are those which consist of non-degenerate quadratic forms.
These orbits are parameterised by their signature and discriminant, so every closed orbit intersects \( \Aut_{\rho(\gG)}(V_\bR) v_0^{(p,q)} \) for some signature~\( (p,q) \).
Thus, in this case, \cref{siegel-intersection} is sufficient to give a polynomial bound for integral elements of a reduced set in every closed orbit.
In general, however, there is no finite subset of \( V_\bR \) whose \( \Aut_{\rho(\gG)}(V_\bR) \)-orbit intersects every closed \( \gG(\bR) \)-orbit, and then \cref{siegel-intersection} does not allow us to compare all closed orbits.

In this example, the representation~\( \rho \) is absolutely irreducible so its only endomorphisms are scalars.
In general, there may be more endomorphisms of~\( \rho \), and it will be important for our applications that we allow \( v \) to be any element of \( \Aut_{\rho(\gG)}(V_\bR) v_0 \), not just a scalar multiple of~\( v_0 \).

\subsection{Fundamental sets for arithmetic groups}\label{subsec:BHC62ii}

The central result of Borel and Harish-Chandra's reduction theory was the construction of fundamental sets for \( \Gamma_\gH \bs \gH(\bR) \), where \( \gH \) is a reductive \( \bQ \)-algebraic group and \( \Gamma_\gH \subset \gH(\bQ) \) is an arithmetic subgroup.
These fundamental sets are constructed by embedding \( \gH \) into some~\( \gGL_n \), where we already know how to obtain fundamental sets using standard Siegel sets.
(Note that \cite[Thm.~6.5]{BHC62} used the notation \( \gG \) where we write \( \gH \) in this theorem.)

\begin{theorem} \cite[Thm.~6.5]{BHC62} \label{bhc-fundamental-sets}
Let \( \gH \) be a reductive \( \bQ \)-algebraic subgroup of \( \gGL_{n,\bQ} \) and let \( \Gamma_\gH = \gGL_n(\bZ) \cap \gH(\bR) \).
Let \( \fS_0 \) be the standard Siegel set in \( \gGL_n(\bR) \).
Let \( u \in \gGL_n(\bR) \) be such that \( u^{-1} \gH(\bR) u \) is self-adjoint.

Then there exists a finite set \( B \subset \gGL_n(\bZ) \) such that
\[ B \fS_0 u^{-1} \cap \gH(\bR) \]
is a fundamental set for \( \Gamma_\gH \) in \( \gH(\bR) \).
\end{theorem}

The ambient group \( \gGL_n \) in \cref{bhc-fundamental-sets} can be replaced by an arbitrary reductive \( \bQ \)-algebraic group~\( \gG \) containg \( \gH \) with only minor alterations to the proof, where \( \fS_0 \) is replaced by a sufficiently large Siegel set in \( \gG(\bR) \).

\subsubsection*{Quantitative fundamental sets for arithmetic groups}

Theorem \ref{fund-set-bound} is a quantitative version of \cref{bhc-fundamental-sets}, where \( \gH \) varies over the \( \bQ \)-algebraic members of a \( \gG(\bR) \)-conjugacy class of subgroups of some fixed reductive group \( \gG \).
This theorem is ``quantitative'' in the sense that it controls a measure of the size of the elements of the finite set~\( B \).
Ideally we would like to bound the height of elements of \( B \) but we have not yet achieved this (it may be possible by combining the methods of this paper with tools of homogeneous dynamics as in \cite{LM16}).
Instead we measure the size of elements of \( B \) in terms of how they act on a vector~\( v_u \) (whose stabiliser is \( \gH \)) in a suitable representation of~\( \gG \).
This turns out to be sufficient for our applications to unlikely intersections.

The theorem will only apply to subgroups~\( \gH \subset \gG \) which are defined over \( \bQ \) because these are the subgroups for which \( \Gamma \cap \gH(\bR) \) is a lattice in \( \gH(\bR) \).
However it is very important that \( \gH \) varies over a \( \gG(\bR) \)-conjugacy class, not just a \( \gG(\bQ) \)-conjugacy class, because this allows the \( \bQ \)-algebraic subgroups in the conjugacy class to belong to more than one isomorphism class over~\( \bQ \).
A striking consequence of allowing this is that the conjugacy class may contain both \( \bQ \)-anisotropic and \( \bQ \)-isotropic groups, so the fundamental set in \( \gH(\bR) \) is sometimes compact and sometimes not compact, yet the same bounds apply to fundamental sets for all \( \gH \) in the conjugacy class. For example \( \gSL_{2,\bQ} \) and unit groups of quaternion algebras can be found in the same \( \gSL_4(\bR) \)-conjugacy class of subgroups of \( \gSL_4 \).

Note also that the semisimple subgroups of \( \gG \) belong to only finitely many \( \gG(\bR) \)-conjugacy classes \cite[Cor.~0.2]{BDR}.
This is not true for reductive subgroups, as may be seen by considering the torus \( \bG_m^2 \), which contains infinitely many non-conjugate subgroups isomorphic to \( \bG_m \) -- see \cite[Remark~1.2]{BDR}.

For an example application of Theorem \ref{fund-set-bound}, consider the case where \( \gG = \gGL_n \) and \( \gH_0 \subset \gGL_n \) is the orthogonal group of the quadratic form \( v_0^{p,q} \).
The representation \( \rho \colon \gG \to \gGL(\Lambda_\bQ) \), where \( \Lambda \) is the $\ZZ$-module of integral quadratic forms of signature \( (p,q) \), satisfies the conditions of \cref{fund-set-bound}.
In particular, (iii) holds because if \( \gH_u \) is defined over~\( \bQ \) then it is the orthogonal group of the integral quadratic form \( v = \rho(u) \lambda v_0^{(p,q)} \) for some \( \lambda \in \bR^\times \).

As noted in \cite[6.7]{BHC62}, the space of Hermite majorants of~\( v \) is 
\[ \Sigma_u = \gH_u(\bR) / (u\gO_n(\bR)u^{-1} \cap \gH_u(\bR)). \]
The image of \( B_u C \fS u^{-1} \cap \gH_u(\bR) \) in \( \Sigma_u \) is a fundamental set for \( \Gamma \cap \gH_u(\bR) \) in \( \Sigma_u \).
\Cref{fund-set-bound} allows us to control the set \( B_u \) used to construct this fundamental set in that for each \( b \in B_u \), the coefficients of the quadratic form \( \rho(b^{-1}) v \) are polynomially bounded in terms of \( \disc(v) \).
A related result can be found in \cite[\S 9.5]{LM16}, which bounds the entries of the matrices \( b \in B_u \) (stronger than bounding \( \rho(b^{-1})v \)), although its bound involves the coefficients of \( v \) as well as the discriminant.

As with \cref{siegel-intersection}, the constants in \cref{fund-set-bound} are ineffective.

\section{Preliminaries} \label{sec:preliminaries}
\subsection{Notation}

If $\Lambda$ is a $\ZZ$-module, we write $\Lambda_\QQ$ for $\Lambda\otimes_\ZZ\QQ$ and $\Lambda_\RR$ for $\Lambda\otimes_\ZZ\RR$. If $\Lambda$ is free and finitely generated and $\phi \colon \Lambda_\QQ \times \Lambda_\QQ \to \QQ$ is a $\QQ$-bilinear form, we denote by $\disc(\Lambda,\phi)$ the determinant of the matrix $(\phi(e_i,e_j))_{i,j}$ where $\{e_1,\ldots,e_n\}$ is a $\ZZ$-basis for $\Lambda$ (the determinant is independent of the choice of basis).

If $R$ is an order in a semisimple $\QQ$-algebra $D$, then we write $\disc(R)$ for the discriminant of the $\ZZ$-module $R$ with respect to the bilinear form $\phi(x,y) = \Tr_{D/\QQ}(xy)$ where $\Tr_{D/\QQ}$ is the (non-reduced) trace of the regular representation of~$D$.  See section~\ref{ssec:discriminants} for more details.

If $V=\Lambda_\QQ$ (or $V=\Lambda_\RR$) and $\gG$ is an algebraic group over $\QQ$ (or $\RR$, respectively), then, for any representation \( \rho \colon \gG \to \gGL(V) \) and \( v \in V \), we write \(\Stab_{\gG,\rho}(v)\) for the stabiliser of \( v \) in \( G \) with respect to \( \rho \), that is,
\[ \Stab_{\gG,\rho}(v) = \{ g \in \gG : \rho(g)v = v \}. \]
If $W$ is a subspace of $V$, we write \(\Stab_{\gG,\rho}(W)\) for the subgroup preserving $W$.
Similarly, we write $\End_{\gG,\rho}(V)$ and $\Aut_{\gG,\rho}(V)$ for the endomorphisms and automorphisms, respectively, of $V$ commuting with $\rho(G)$, and we also write $\End_{\gG,\rho}(\Lambda)$ for the endomorphisms of $\Lambda$ commuting with $\rho(G)$. If $\rho$ is an inclusion $\gG \hookrightarrow \gGL(V)$, then we omit it from the subscripts. If $D$ is a ring acting on $V$, we denote by $\End_D(V)$ the endomorphisms of $V$ commuting with the action of $D$. 

We denote by $\gG^\der$ the derived subgroup of $\gG$, by $Z(\gG)$ the centre of $\gG$, and by $\gG(\bR)^+$ the connected component of $\gG(\bR)$ (in the archimedean topology) containing the identity. If $\gS$ is a split $\QQ$-subtorus of $\gG$, we denote by $Z_\gG(\gS)$ the centraliser of $\gS$ in $\gG$ and by $X^*(\gS)$ the character group of $\gS$.

If $V=\Lambda_\RR$, we write \( \abs{\cdot} \) for a norm on \( V \).
Unless otherwise specified, it does not matter which norm we choose, except that the values of constants will depend on the norm.
Whenever the statement of a theorem involves a norm \( \abs{\cdot} \), we implicitly assume that such a norm has been chosen, and the constants in the theorem implicitly depend on this choice. We write \( \length{\cdot} \) for the associated operator norm on \( \End_\bR(V) \).
In other words, for \( f \in \End_\bR(V) \),
\[ \length{f} = \sup \{ \abs{f(v)} : v \in V, \abs{v} = 1 \}. \]

\subsection{Cartan involutions}

The theory of Cartan involutions is well-known for connected semisimple groups.
However, for reductive real algebraic groups, several definitions of Cartan involutions are used in the literature. The following seems to us to be the most elegant definition.

\begin{definition}
Let \( \gG \) be a reductive \( \bR \)-algebraic group.
A \defterm{Cartan involution} of~\( \gG \) is an involution \( \theta \colon \gG \to \gG \) in the category of \( \bR \)-algebraic groups such that the set of fixed points of \( \theta \) in \( \gG(\bR) \) is a maximal compact subgroup of \( \gG(\bR) \).
\end{definition}

The fundamental example is the standard Cartan involution \( x \mapsto (x^t)^{-1} \) on \( \gGL_n \), whose real fixed point set is \( \gO_n(\bR) \).


The commonly used definitions of Cartan involutions for reductive real algebraic groups are all equivalent to this one, but the equivalences are not obvious and it is difficult to find proofs for all of the equivalences.
For convenience, we provide a list of equivalent definitions, and we will post a self-contained proof of this lemma on arXiv.
In the following lemma: (ii) is the definition of Cartan involution used in \cite[11.17]{Bor69}; 
(iii) is the definition used in \cite{BHC62} and \cite{RS90}; while (iv) is the definition from \cite[p.~255]{Del79}, commonly used in the study of Shimura varieties.

\begin{lemma}
Let \( \gG \) be a reductive \( \bR \)-algebraic group and let \( \theta \colon \gG(\bR) \to \gG(\bR) \) be an involution in the category of real Lie groups.
Let \( \gZ_d \) denote the maximal \( \bR \)-split torus in the centre of \( \gG \).
The following are equivalent:
\begin{enumerate}[(i)]
\item \( \theta \) is a Cartan involution as defined above;
\item the set of fixed points of \( \theta \) in \( \gG(\bR) \) is a maximal compact subgroup of \( \gG(\bR) \) and \( \theta(z) = z^{-1} \) for all \( z \in \gZ_d(\bR) \);
\item there exists a faithful representation \( \rho \colon \gG \to \gGL_{n,\bR} \) in the category of \( \bR \)-algebraic groups such that
\[ \rho(\theta(g)) = (\rho(g)^t)^{-1} \]
for all \( g \in \gG(\bR) \);
\item \( \theta \) is a morphism in the category of \( \bR \)-algebraic groups and the real form \( G^\theta = \{ g \in \gG(\bC) : \theta(\bar g) = g \} \) is compact and intersects every connected component of \( \gG(\bC) \), where \( \bar{\cdot} \) denotes complex conjugation.
\end{enumerate}
Furthermore, for each maximal compact subgroup \( K \subset \gG(\bR) \), there is a unique Cartan involution of \( \gG \) whose set of real fixed points is \( K \).
\end{lemma}

Given a reductive \( \bR \)-algebraic group \( \gG \) and a Cartan involution \( \theta \) of \( \gG \), we say that an algebraic subgroup \( \gH \subset \gG \) is \defterm{self-adjoint} (with respect to \( \theta \)) if \( \theta(\gH) = \gH \).

In several of our theorem statements (including \cref{siegel-intersection,fund-set-bound}), we are given a reductive \( \bQ \)-algebraic group \( \gG \) and a Siegel set \( \fS \subset \gG(\bR) \).
It will be seen in section~\ref{ssec:siegel-sets} that the definition of a Siegel set involves the choice of a maximal compact subgroup \( K \subset \gG(\bR) \).
In such a situation, we say that a subgroup of \( \gG \) is self-adjoint if it is self-adjoint with respect to the Cartan involution whose fixed point set is the \( K \) used in the construction of the Siegel set.

\subsection{Siegel sets} \label{ssec:siegel-sets}

We use the definition of Siegel sets from \cite[sec.~2.2]{Orr18}, which is a minor modification of definitions used in \cite[Def.~12.3]{Bor69} and \cite[Ch.~II, sec.~4.1]{AMRT}.
For a comparison between these definitions, see \cite[sec.~2.3]{Orr18}.

Let \( \gG \) be a reductive \( \bQ \)-algebraic group.
In order to define a Siegel set in \( \gG(\bR) \), we begin by making choices of the following subgroups of~\( \gG \):

\begin{enumerate}
\item \( \gP \) a minimal parabolic \( \bQ \)-subgroup of \( \gG \);
\item \( K \) a maximal compact subgroup of \( \gG(\bR) \).
\end{enumerate}

As a consequence of \cite[Ch.~II, Lemma~3.12]{AMRT}, there is a unique \( \bR \)-torus \( \gS \subset \gP \) satisfying the following conditions:
\begin{enumerate}[(i)]
\item \( \gS \) is \( \gP(\bR) \)-conjugate to a maximal \( \bQ \)-split torus in \( \gP \).
\item \( \gS \) is self-adjoint with respect to the Cartan involution associated with \( K \).
\end{enumerate}
These conditions could equivalently be stated as:
\begin{enumerate}[(i)]
\item \( \gS \) is a lift of the unique maximal \( \bQ \)-split torus in \( \gP/\gU \), where $\gU$ denotes the unipotent radical of $\gP$.
\item \( \operatorname{Lie} \gS(\bR) \) is orthogonal to \( \operatorname{Lie} K \) with respect to the Killing form of~\( \gG \).
\end{enumerate}

Define the following further pieces of notation:
\begin{enumerate}
\item \( \gM \) is the preimage in \( Z_\gG(\gS) \) of the maximal \( \bQ \)-anisotropic subgroup of~\( \gP/\gU \).
(Note that by \cite[Corollaire~4.16]{BT65}, \( Z_\gG(\gS) \) is a Levi subgroup of \( \gP \) and hence maps isomorphically onto \( \gP/\gU \).)
\item \( \Delta \) is the set of simple roots of \( \gG \) with respect to \( \gS \), using the ordering induced by \( \gP \).
(The roots of \( \gG \) with respect to \( \gS \) form a root system because \( \gS \) is conjugate to a maximal \( \bQ \)-split torus in~\( \gG \).)
\item \( A_t = \{ \alpha \in \gS(\bR)^+ : \chi(\alpha) \geq t \text{ for all } \chi \in \Delta \} \) for any real number \( t > 0 \).
\end{enumerate}

A \defterm{Siegel set} in \( \gG(\bR) \) (with respect to \( (\gP, \gS, K) \)) is a set of the form
\[ \fS = \Omega A_t K \]
where
\begin{enumerate}
\item \( \Omega \) is a compact subset of \( \gU(\bR) \gM(\bR)^+ \); and
\item \( t \) is a positive real number.
\end{enumerate}

We say that a set \( \Omega \subset \gG(\bR) \) is a \defterm{fundamental set} for \( \Gamma \) if the following conditions are satisfied:
\begin{enumerate}[(F1)]
\item \( \Gamma \Omega = \gG(\bR) \); and
\item for every \( g \in \gG(\bQ) \), the set
\( \{ \gamma \in \Gamma : \gamma \Omega \cap g \Omega \neq \emptyset \} \)
is finite (the \defterm{Siegel property}).
\end{enumerate}

The following two theorems show that, if we make suitable choices of Siegel set \( \fS \subset \gG(\bR) \) and finite set \( C \subset \gG(\bQ) \), then \( C \fS \) is a fundamental set for \( \Gamma \) in \( \gG(\bR) \).

\begin{theorem} \cite[Th\'eor\`eme 13.1]{Bor69} \label{borel-siegel-surj}
Let \( \Gamma \) be an arithmetic subgroup of \( \gG(\bQ) \).
For any minimal parabolic \( \bQ \)-subgroup \( \gP \subset \gG \) and maximal compact subgroup \( K \subset \gG(\bR) \),
there exist a Siegel set~\( \fS \subset \gG(\bR) \) with respect to \( (\gP, \gS, K) \) and a finite set \( C \subset \gG(\bQ) \) such that
\[ \gG(\bR) = \Gamma C \fS. \]
\end{theorem}

\begin{theorem} \cite[Th\'eor\`eme 15.4]{Bor69} \label{borel-siegel-finite}
Let \( \Gamma \) be an arithmetic subgroup of \( \gG(\bQ) \).
Let \( \fS \subset \gG(\bR) \) be a Siegel set.
For any finite set \( C \subset \gG(\bQ) \) and any element \( g \in \gG(\bQ) \), the set
\[ \{ \gamma \in \Gamma : \gamma C \fS \cap g C \fS \neq \emptyset \} \]
is finite.
\end{theorem}

A quantitative version of \cref{borel-siegel-finite} can be found at \cite[Thm.~1.1]{Orr18}.

\section{Quantitative reduction theory} \label{sec:qrt}

In this section, we will prove \cref{siegel-intersection,fund-set-bound}.
The proof follows the same strategy as that of \cite[Lemma~5.3 and Thm.~6.5]{BHC62}.
We replace the purely topological proof of \cite[Prop.~5.2]{BHC62} by an orbit growth bound of Eberlein \cite{Ebe14} using Riemannian geometry in \( \gGL_n(\bR)^+ \).
We also prove a lemma bounding the norm of \( \tau \in \Aut_{\rho(\gG)}(V_\bR) \) in terms of the length of \( v = \tau(v_0) \) -- this is a calculation in a semisimple \( \bR \)-algebra.
For the rest, the proof closely follows the method of Borel and Harish-Chandra, keeping track of quantitative information and the action of \( \Aut_{\rho(\gG)}(V_\bR) \) throughout and with some small adaptations to generalise to reductive groups whose \( \bR \)-rank is greater than their \( \bQ \)-rank.

\subsection{Bound for orbits of real reductive groups}

We begin by proving the following bound for orbits in representations of real reductive groups, not yet considering any arithmetic subgroup.

\begin{proposition} \label{riemannian-orbit-bound}
Let \( \gG \) be a reductive \( \bR \)-algebraic group and let \( \rho \colon \gG \to \gGL(V_\bR) \) be an \( \bR \)-algebraic representation.
Let \( v_0 \in V_\bR \) be a non-zero vector whose orbit \( \rho(\gG(\bR))v_0 \) is closed.
Then there exist constants \( \newC{orbit-bound-multiplier} \) and \( \newC{orbit-bound-exponent} \) (depending on \( \gG \), \( \rho \) and \( v_0 \)) such that, for every \( w \in \rho(\gG(\bR))v_0 \), there exists \( g \in \gG(\bR) \) satisfying \( w = \rho(g)v_0 \) and
\[ \max(\length{\rho(g)}, \length{\rho(g^{-1})})  \leq  \refC{orbit-bound-multiplier} \abs{w}^{\refC{orbit-bound-exponent}}. \]
\end{proposition}

\Cref{riemannian-orbit-bound} provides a quantitative version of \cite[Prop.~5.2]{BHC62}, which asserts that if \( w \in \rho(\gG(\bR))v_0 \cap Q \) for some compact subset \( Q \subset V_\bR \), then in fact \( w \in \rho(\Omega)v_0 \) for some compact subset \( \Omega \subset \gG(\bR) \) (independent of $w$).
Here we show that the operator norm of elements of \( \rho(\Omega) \) is polynomially bounded with respect to the length of vectors in \( Q \).

We define a Riemannian metric on $\gGL_n(\bR)^+$ as follows.
The positive definite bilinear form \( (A, B) \mapsto \tr(AB^t) \) on \( \rM_n(\bR)\), which is the Lie algebra of \(\gGL_n(\bR) \), induces a right-invariant Riemannian metric on the Lie group \( \gGL_n(\bR)^+ \).
Let \( d_\mathrm{R} \) denote the distance function on \( \gGL_n(\bR)^+ \) induced by this Riemannian metric.

Eberlein's theorem relates \( \abs{\rho(g)v_0} \) to the Riemannian distance between \( g \) and the stabiliser of \( v_0 \).
We will combine this with the following lemma bounding \( \length{\rho(g)} \) in terms of the Riemannian distance.

\begin{lemma} \label{metric-inequality}
Let \( I \) denote the identity matrix in \( \gGL_n(\bR) \).
There exists a constant \( \newC{metric-multiplier}(n) \) such that
every \( g \in \gGL_n(\bR)^+ \) satisfies
\[ \length{g} \leq \refC{metric-multiplier}(n) \, \exp(d_\mathrm{R}(g, I)). \]
\end{lemma}

\begin{proof}
Let \( \abs{g}_\mathrm{F} \) denote the Frobenius norm of \( g \), that is,
\[ \abs{g}_\mathrm{F} = \sqrt{\tr(gg^t)}. \]

Using the Cartan decomposition, we can write \( g = k \exp(X) \) for some \( k \in \gSO_n(\bR) \) and some symmetric matrix~\( X \in \rM_n(\bR) \).
Let \( \lambda_{\max} \) denote the largest eigenvalue of~\( X \) (note that \( X \) is diagonalisable and all its eigenvalues are real because it is symmetric).

By \cite[Prop.~4.8]{Ebe14}, we have
\begin{equation} \label{eqn:metric-inequality}
n^{-1/2} \exp(-\refC{eberlein-exp}(n)) \exp(\abs{X}_\mathrm{F} - \lambda_{\max}) \leq \frac{\exp(d_\mathrm{R}(g, I))}{\abs{g}_\mathrm{F}}
\end{equation}
for some constant~\( \newC{eberlein-exp}(n) \) which depends only on~\( n \).
Since \( X \) is symmetric, we have
\[ \abs{X}_\mathrm{F} = \sqrt{\tr(XX^t)} = \sqrt{\tr(X^2)} = \sqrt{\sum_{i=1}^n \lambda_i^2} \geq \lambda_{\max} \]
where \( \lambda_1, \dotsc, \lambda_n \) denote the eigenvalues of \( X \).
Hence \( \exp(\abs{X}_\mathrm{F} - \lambda_{\max}) \geq 1 \), so \eqref{eqn:metric-inequality} implies that
\[ \abs{g}_\mathrm{F} \leq \refC{metric-multiplier-F}(n) \, \exp(d_\mathrm{R}(g, I)) \]
where \( \newC{metric-multiplier-F}(n) = n^{-1/2} \exp(-\refC{eberlein-exp}(n)) \).

Since \( \abs{\cdot}_\mathrm{F} \) and \( \length{\cdot} \) are norms on the finite-dimensional vector space \( \rM_n(\bR) \), they are equivalent, so this proves the lemma.
\end{proof}

\begin{proof}[Proof of \cref{riemannian-orbit-bound}]
Let \( G = \gG(\bR)^+ \).
Fix a finite list of representatives \( a_1, \dotsc, a_r \) for the connected components of \( \gG(\bR) \).
Then, given \( w \in \rho(\gG(\bR))v_0 \), we can write \( w = a_i w' \) for some \( i \leq r \) and some \( w' \in \rho(G)v_0 \).
Hence it suffices to prove the proposition for \( w \in \rho(G)v_0 \).

By \cite{Mos55} we can choose an inner product on \( V_\bR \) with respect to which \( \rho(G) \) is self-adjoint.
Since all norms on a finite-dimensional vector space are equivalent, it suffices to prove the proposition under the assumption that the norm on~\( V_\bR \) is induced by such an inner product.
In particular, the stabiliser of the norm in \( G \) is a maximal compact subgroup~\( K \) and if \( \fp \) denotes the \( -1 \) eigenspace of the associated Cartan involution~\( \theta \) on \( \Lie(G) \), then for every \( X \in \fp \), \( \mathrm{d}\rho(X) \) is self-adjoint.
Thus the conditions of \cite[sec.~3]{RS90} are satisfied.

Since \( \rho(G)v_0 \) is closed it contains a minimal vector, that is, a vector whose length is minimal among all elements of the orbit. (Indeed, a theorem of Richardson and Slodowy \cite[Thm.~4.4]{RS90} states that the two properties are, in fact, equivalent).
Replacing \( v_0 \) by another vector in its orbit changes the element \( g \) such that \( w = \rho(g)v_0 \) by a fixed element of \( G \), so we may assume that \( v_0 \) itself is a minimal vector.

Let \( H_0 = \Stab_{\rho(G)}(v_0) \subset \gGL_n(\bR) \). Note that \( H_0 \) is self-adjoint with respect to our chosen inner product on \( V_\bR \) (see \cite[Thm.~4.3]{RS90}, for example).

If \( \rho(G)v_0 \) is bounded, then it is compact, so by \cite[Prop.~5.2]{BHC62} there exists a compact set \( \Omega \subset G \) such that \( \rho(G)v_0 = \rho(\Omega)v_0 \).
The elements \( g \in \Omega \) satisfy a uniform bound for \( \max(\length{\rho(g)}, \length{\rho(g^{-1})}) \), proving the proposition in this case since $|w|\geq |v_0|>0$.


From now on, assume that the orbit \( \rho(G)v_0 \) is unbounded.
Then \cite[sec.~2.5]{Ebe14} defines an associated value \( \lambda^-(v_0) \in \bR \).
By \cite[Thm.~3.1 (1)]{Ebe14}, since \( v_0 \) is minimal, \( \lambda^-(v_0) > 0 \).
By \cite[Thm.~3.1 (2)]{Ebe14}, we also have
\[ \lambda^-(v_0) \leq \liminf_{d_\mathrm{R}(\rho(g), H_0) \to \infty} \frac{\log\abs{\rho(g)v_0}}{d_\mathrm{R}(\rho(g), H_0)}. \]
Hence there exists a constant \( \newC{dR-liminf} > 0 \) (depending only on \( G \), \( \rho \) and \( v_0 \)) such that
\[ \frac{\log \abs{\rho(g)v_0}} {d_\mathrm{R}(\rho(g), H_0)} > \half \lambda^-(v_0) \]
for all \( g \in G \) satisfying \( d_\mathrm{R}(\rho(g), H_0) > \refC{dR-liminf} \).

On the other hand, if \( d_\mathrm{R}(\rho(g), H_0) \leq \refC{dR-liminf} \), then because \( v \) is minimal we have
\[ \frac{\log \abs{\rho(g)v_0}} {d_\mathrm{R}(\rho(g), H_0)} \geq \frac{\log \abs{v_0}} {d_\mathrm{R}(\rho(g), H_0)} \geq \frac{\log \abs{v_0}} {\refC{dR-liminf}} \]
which is a positive constant.

Combining the above two inequalities, we deduce that there is a positive constant~\( \newC{eberlein-const} \) such that the following inequality holds for all \( g \in G \):
\[ \frac{\log \abs{\rho(g)v_0}} {d_\mathrm{R}(\rho(g), H_0)} \geq \refC{eberlein-const} \]
or in other words,
\begin{equation} \label{eqn:eberlein-inequality}
\abs{\rho(g)v_0}^{1/\refC{eberlein-const}} \geq \exp(d_\mathrm{R}(\rho(g), H_0)).
\end{equation}

Given \( w \in \rho(G)v_0 \), write \( w = \rho(g')v_0 \) where \( g' \in G \).
Since \( H_0 \) is closed, we can choose \( h \in H_0 \) such that \( d_\mathrm{R}(\rho(g'), H_0) = d_\mathrm{R}(\rho(g'), h) \).

Since \( H_0 \subset \rho(G) \), we can choose \( g \in G \) such that \( \rho(g) = \rho(g')h^{-1} \).
Since \( h \in H_0 \), we have \( \rho(g)v_0 = w \).
Since \( d_\mathrm{R} \) is right invariant, we have
\[ d_\mathrm{R}(\rho(g'), h) = d_\mathrm{R}(\rho(g), I) = d_\mathrm{R}(I, \rho(g^{-1})). \]
Thus \eqref{eqn:eberlein-inequality} (applied to \( g' \)) becomes
\[ \abs{w}^{1/\refC{eberlein-const}} \geq \exp(d_\mathrm{R}(\rho(g), I)) = \exp(d_\mathrm{R}(\rho(g^{-1}), I)). \]
Applying \cref{metric-inequality} to both \( \rho(g) \) and \( \rho(g^{-1}) \) completes the proof of the proposition.
\end{proof}

\begin{remark} \label{ineffectivity}
The proof of \cref{riemannian-orbit-bound} is ineffective for two reasons.
\begin{enumerate}
\item It depends on the value of \( \lambda^-(v_0) \).
We do not know a general method for calculating this value, although it seems to be feasible to calculate it in particular cases.
\item The value \( \refC{dR-liminf} \) depends on the speed of convergence of the limit in \cite[Thm.~3.1]{Ebe14}, which is ineffective.
\end{enumerate}
\end{remark}

\subsection{Quantitative reduction theory for representations}

We now prove \cref{siegel-intersection}.
The proof follows that of \cite[Lemma~5.3]{BHC62}, keeping track of quantitative information and some minor generalisations.
For the sake of clarity, we have broken it down into a series of lemmas, each proved by a short calculation.

We use the notation for Siegel sets from section~\ref{ssec:siegel-sets}.
Then \( \gS \) is \( \gP(\bR) \)-conjugate to a maximal \( \bQ \)-split torus \( \gT \) in \( \gP \).
Since \( \gP(\bR) = Z_\gG(\gS)(\bR).\gU(\bR) \), we can choose \( n \) in \( \gU(\bR) \) such that \( \gS = n\gT n^{-1} \).

We also adopt some notation from the proof of \cite[Lemma~5.3]{BHC62} (bearing in mind that we have reversed the order of multiplication in our Iwasawa decomposition relative to \cite{BHC62}).
By the Iwasawa and Langlands decompositions, the multiplication map
\[ \gU(\bR) \times \gS(\bR)^+ \times \gM(\bR)^+K \to \gG(\bR) \]
is bijective.
Given \( x \in \gG(\bR) \), we write it as \( x = n_x a_x k_x \) according to this decomposition.
Let
\[ y_x = a_x^{-1}n x, \quad z_x = a_x^{-2}n x. \]

For each character \( \chi \in X^*(\gS) \), let \( V_\chi \) denote the corresponding eigenspace in~\( V_\bR \).
We have \( V_\bR = \bigoplus_\chi V_\chi \) and we let \( \pi_\chi \colon V \to V_\chi \) denote the projection maps.
Since all norms on the finite-dimensional vector space \( V_\bR \) are equivalent, we may assume without loss of generality that the norm is chosen so that the spaces \( V_\chi \) are orthogonal to each other.

In the lemmas which follow, \( \tau \) denotes an element of \( \Aut_{\rho(\gG)}(V_\bR) \).
Constants labelled \( c_n \) depend only on \( \gG \), \( \fS \), \( \rho \), \( \Lambda \) and \( v_0 \), and not on \( \tau \), \( x \), \( v \) or~\( w \).

\begin{lemma} \label{lattice-lower-bound}
There exists a constant \( \newC{lattice-lower-bound} > 0 \) such that for all \( v' \in V_\bR \) and all \( \chi \in X^*(\gS) \), if \( \tau v' \in \Lambda \), then either \( \pi_\chi(\rho(n)v') = 0 \) or \( \abs{\pi_\chi(\rho(n)v')} \geq \refC{lattice-lower-bound} / \length{\tau} \).
\end{lemma}

\begin{proof}
Since $\gT$ is $\bQ$-split, its eigenspaces $V_{\psi}$ are defined over $\bQ$ and $V$ decomposes as $\oplus_{\psi\in X^*(\gT)}V_{\psi}$. For $\psi\in X^*(\gT)$, let $\pi_{\psi}$ denote the projection $V\to V_{\psi}$ in this direct sum. Because the $V_\psi$ are defined over $\bQ$, the image $\pi_{\psi}(\Lambda)$ is a lattice in $V_{\psi}$. Hence there is a constant \( \newC{lattice-lower-bound-aux} > 0 \) such that, if \( \tau v' \in \Lambda \), then 
\[ \pi_{\psi}(\tau v') = 0 \text{ or } \abs{\pi_{\psi}(\tau v')} \geq \refC{lattice-lower-bound-aux}. \]
Since \( V_{\psi} \neq 0 \) for only finitely many characters~\( \psi\in X^*(\gT)\), it is possible to choose a single constant \( \refC{lattice-lower-bound-aux} > 0 \) which works for every~\( \psi \).

For each $\chi\in X^*(\gS)$, the eigenspace $V_\chi$ of $\gS$ is equal to $\rho(n)V_{\psi}$ for some $\psi\in X^*(\gT)$. It follows that $\pi_\chi=\rho(n)\circ\pi_{\psi}\circ\rho(n)^{-1}$.
Therefore there is a constant \( \refC{lattice-lower-bound} > 0 \) (namely, $\refC{lattice-lower-bound-aux}\length{\rho(n)^{-1}}^{-1}$) such that, if \( \tau v' \in \Lambda \), then 
\[ \pi_\chi(\tau\rho(n)v') = 0 \text{ or } \abs{\pi_\chi(\tau\rho(n)v')} \geq \refC{lattice-lower-bound}. \]

Since \( \tau \) commutes with \( \rho(\gG(\bR)) \), it preserves the eigenspaces \( V_\chi \) and hence commutes with \( \pi_\chi \). Therefore, either
\[\tau(\pi_\chi(\rho(n)v'))=\pi_\chi(\tau\rho(n)v')=0,\]
which implies $\pi_\chi(\rho(n)v'))=0$, or
\[ \length{\tau} \abs{\pi_\chi(\rho(n)v')} \geq \abs{\tau(\pi_\chi(\rho(n)v'))} = \abs{\pi_\chi(\tau\rho(n)v')} \geq \refC{lattice-lower-bound}.
\qedhere \]
\end{proof}

\begin{lemma} \label{yxv-bound}
There exists a constant \( \newC{yxv-bound} \) such that, for all \( x \in \fS \), we have
\[ \abs{\rho(y_x)v_0} \leq \refC{yxv-bound}. \]
\end{lemma}

\begin{proof}
From the definition of a Siegel set, \( \{ n_x : x \in \fS \} \) is relatively compact. Hence, \( \{ n n_x : x \in \fS \} \) is a relatively compact subset of $\gU(\RR)$.
Therefore, by \cite[Lemme~12.2]{Bor69}, \( \{ a_x^{-1} nn_x a_x : x \in \fS \} \) is relatively compact.
Furthermore, \( \{ k_x : x \in \fS \} \) is also relatively compact.
Since
\[ y_x = a_x^{-1} nx = a_x^{-1} nn_x a_x k_x \]
we conclude that \( \{ y_x : x \in \fS \} \) is relatively compact.
\end{proof}

\begin{lemma} \label{zxv-bound}
There exists a constant \( \newC{zxv-multiplier} \) such that, for all \( x \in \fS \), if \( \tau\rho(x)v_0\in \Lambda \), then
\[ \abs{\rho(z_x)v_0}  \leq  \refC{zxv-multiplier} \length{\tau}. \]
\end{lemma}

\begin{proof}
Let \( \chi \in X^*(\gS) \).
From the definitions of \( y_x \) and \( z_x \), we can calculate
\[ \pi_\chi(\rho(y_x)v_0) = \chi(a_x)^{-1} \pi_\chi(\rho(nx)v_0),  \quad  \pi_\chi(\rho(z_x)v_0) = \chi(a_x)^{-2} \pi_\chi(\rho(nx)v_0). \]
Therefore, either \( \pi_\chi(\rho(nx)v_0) = 0 \), in which case \( \pi_\chi(\rho(z_x)v_0) = 0 \), or else, by \cref{lattice-lower-bound} (applied to \( v' = \rho(x)v_0 \)) and \cref{yxv-bound}, we have
\[ \abs{\pi_\chi(\rho(z_x)v_0)}
   = \frac{\abs{\pi_\chi(\rho(y_x)v_0)}^2}{\abs{\pi_\chi(\rho(nx)v_0)}}
   \leq  \frac{\refC{yxv-bound}^2} {\refC{lattice-lower-bound} / \length{\tau}}
   = \newC* \length{\tau}. \]

Since \( V_\bR \) is the orthogonal direct sum of the \( V_\chi \), the lemma follows by squaring and summing over \( \chi  \).
\end{proof}

\begin{lemma} \label{op-norm-bound}
There exist constants \( \newC{op-norm-multiplier} \) and \( \newC{op-norm-exponent} \) such that, for every \( x \in \fS \), if \( \tau(\rho(x)v_0) \in \Lambda \), then
there exists \( g \in \gG(\bR) \) satisfying
\[ \rho(g)v_0 = \rho(a_x^{-1}k_x)v_0  \text{ and }  \max(\length{\rho(g)}, \length{\rho(g)^{-1}})  \leq  \refC{op-norm-multiplier} \length{\tau}^{\refC{op-norm-exponent}}. \]
\end{lemma}

\begin{proof}
By \cref{riemannian-orbit-bound,zxv-bound}, there exists \( g' \in \gG(\bR) \) such that \( \rho(z_x)v_0 = \rho(g')v_0 \) and
\begin{equation} \label{eqn:g'-op-norm-bound}
   \max(\length{\rho(g')}, \length{\rho(g'^{-1})})
      \leq  \refC{orbit-bound-multiplier} \abs{\rho(z_x)v_0}^{\refC{orbit-bound-exponent}}
      \leq  \refC{orbit-bound-multiplier}\refC{zxv-multiplier}^{\refC{orbit-bound-exponent}} \length{\tau}^{\refC{orbit-bound-exponent}}.
\end{equation}
Let
\[ g = a_x^{-2} n_x^{-1}n^{-1} a_x^{2} \, g'. \]
Then
\begin{align*}
    \rho(g)v_0
  & = \rho(a_x^{-2} n_x^{-1}n^{-1} a_x^{2}) \rho(g') v_0 = \rho(a_x^{-2} n_x^{-1}n^{-1} a_x^{2}) \rho(z_x) v_0
\\& = \rho(a_x^{-2} n_x^{-1} x) v_0 = \rho(a_x^{-1} k_x) v_0.
\end{align*}
Meanwhile, by \cite[Lemma~12.2]{Bor69}, \( \{ a_x^{-2} n^{-1}_x n a_x^2 : x \in \fS \} \) is relatively compact so \eqref{eqn:g'-op-norm-bound} implies the required bound on \( \max(\length{\rho(g)}, \length{\rho(g)^{-1}}) \).
\end{proof}

Let \( \theta \) denote the Cartan involution of \( \gG \) whose set of real fixed points is~\( K \).

\begin{lemma} \label{cartan-manip}
There exists a compact set \( \Phi \subset \gG(\bR) \) such that, for all \( x \in \fS \), we have \( \theta(a_x^{-1} k_x) \in \Phi x \).
\end{lemma}

\begin{proof}
By definition \( k_x \in \gM(\bR)^+K \), so we can write \( k_x = m_x \ell_x \) where \( m_x \in \gM(\bR)^+ \) and \( \ell_x \in K \).
This is not a unique decomposition, but the definition of Siegel set guarantees that we can choose \( m_x \) in a fixed compact subset of \( \gM(\bR)^+ \).
(Recall that, by definition, \( \gM \) commutes with \( \gS \).)

By definition, \( \theta \) acts trivially on \( K \) and stabilises \( \gS(\bR) \).
Since \( \gS \) is an \( \bR \)-split torus, the latter implies that \( \theta(a) = a^{-1} \) for all \( a \in \gS(\bR) \).
Hence
\[ \theta(a_x^{-1} k_x) = \theta(m_x a_x^{-1} \ell_x) = \theta(m_x) a_x \ell_x = \theta(m_x) m_x^{-1} n^{-1}_x x. \]
Since \( m_x \) and \( n_x \) lie in compact sets independent of \( x \), this proves the lemma.
\end{proof}

We are now ready to prove a version of \cref{siegel-intersection} in which the bound is expressed in terms of the operator norm of \( \tau \in \Aut_{\rho(\gG)}(V_\bR) \), instead of the length of \( v = \tau(v_0) \).

\begin{proposition} \label{siegel-intersection-tau}
Let \( \gG \) be a reductive \( \bQ \)-algebraic group and let \( \fS \subset \gG(\bR) \) be a Siegel set.
Let \( \rho \colon \gG \to \gGL(V) \) be a representation of \( \gG \) defined over \( \bQ \).
Let \( \Lambda \subset V \) be a \( \bZ \)-lattice.
Let \( v_0 \in V_\bR \) be such that:
\begin{enumerate}[(i)]
\item \( \rho(\gG(\bR))v_0 \) is closed in \( V_\bR \);
\item the stabiliser \( \Stab_{\gG(\bR),\rho} (v_0) \) is self-adjoint.
\end{enumerate}

Then there exist constants \( \newC{siegel-intersection-tau-multiplier} \), \( \newC{siegel-intersection-tau-exponent} \) such that,
for every \( \tau \in \Aut_{\rho(\gG)}(V_\bR) \) and every \( w \in \rho(\fS) \tau(v_0) \cap \Lambda \), we have \( \abs{w} \leq \refC{siegel-intersection-tau-multiplier} \length{\tau}^{\refC{siegel-intersection-tau-exponent}} \).
\end{proposition}

\begin{proof}
Write \( w = \rho(x)\tau(v_0) = \tau(\rho(x)v_0) \), with \( x \in \fS \).
Then we get \( g \) as in \cref{op-norm-bound}.

By \cite[Prop.~13.5]{BHC62}, there is a Cartan involution \( \theta' \) of \( \gGL(V_\bR) \) such that \( \theta' \circ \rho = \rho \circ \theta \).
With respect to a suitable basis of \( V_\bR \), \( \theta' \) is given by \( g \mapsto (g^{-1})^t \).
The norms \( \length{X} \) and \( \length{X^t} \) on the finite dimensional vector space \( \End(V_\bR) \) are equivalent, so there exists a constant \( \newC{norm-transpose-multiplier} \) such that \( \length{\rho(\theta(g))} \leq \refC{norm-transpose-multiplier} \length{\rho(g^{-1})} \).

We have \( g = a_x^{-1} k_x h \) where \( h \in H_0 = \Stab_{\gG(\bR), \rho}(v_0) \).
Hence by \cref{cartan-manip}, we get
\[ \theta(g) = \theta(a_x^{-1} k_x) \theta(h) \in \Phi x \theta(h) \]
where \( \Phi \) is a fixed compact set.
Hence we get
\[ \length{\rho(x\theta(h))} \leq \newC* \length{\rho(\theta(g))} \leq \newC* \length{\rho(g^{-1})} \leq \newC* \length{\tau}^{\newC*}. \]

By hypothesis, \( H_0 \) is self-adjoint so \( \theta(h) \in H_0 \).
Hence \( \rho(x\theta(h)) v_0 = \rho(x)v_0 \) so \( w = \tau(\rho(x\theta(h)v_0) \) and
\[ \abs{w} \leq \length{\tau} \length{\rho(x\theta(h))} \abs{v_0} \]
which is polynomially bounded with respect to \( \length{\tau} \), as required.
\end{proof}

To conclude, we show that it is possible to choose \( \tau \) such that $\tau(v_0)=v$ and \( \length{\tau} \) is bounded in terms of \( \abs{v} \).
\Cref{siegel-intersection} follows by combining \cref{siegel-intersection-tau} with \cref{tau-bound}, applied to \( E = \End_{\rho(\gG)}(V_\bR) \).

\begin{lemma} \label{tau-bound}
Let \( V_\bR \) be a real vector space and let \( E \) be a semisimple \( \bR \)-subalgebra of \( \End(V_\bR) \).
Let \( v_0 \in  V_\bR \).

Then there exists a constant \( \newC{tau-mult} \) such that, for every \( v \in E^\times v_0 \), there exists \( e \in E^\times \) satisfying \( v = ev_0 \) and \( \length{e} \leq \refC{tau-mult} \abs{v} \).
\end{lemma}

Note that \( E^\times \) is the group of \( \bR \)-points of a reductive \( \bR \)-algebraic group. However, this lemma does not follow from \cref{riemannian-orbit-bound} because the orbit \( E^\times v_0 \) is not closed. 

\begin{proof}[Proof of \cref{tau-bound}]
Write \( E \) as a product of simple \( \bR \)-algebras \( \prod_{i=1}^m E_i \).
There is a corresponding decomposition \(  V_\bR = \bigoplus_{i=1}^m V_i \), where the action of \( E_i \) on \(  V_\bR \) factors through~\( V_i \).
If \( v_0 = \sum_{i=1}^m v_i \in  V_\bR \) and \( e = (e_1, \dotsc, e_m) \in E \), then \( ev_0 = \sum_{i=1}^m e_i v_i \).

Because all norms on a finite-dimensional real vector space are equivalent, we may assume without loss of generality that the norm of each element of \( V_\bR \) is the maximum of the norms of its projections to the \( V_i \).  Then the operator norm satisfies \( \length{e} = \max\{ \length{e_i} : 1 \leq i \leq m \} \).
Hence it suffices to prove the lemma for each pair \( (E_i, V_i) \).
In other words, we may assume that \( E \) is a simple \( \bR \)-algebra.

Then \( E = \rM_n(D) \) for some positive integer~\( n \), where \( D \) is a division algebra isomorphic to \( \bR \), \( \bC \) or \( \bH \).  There is a unique simple \( E \)-module, namely \( D^n \), such that
we can identify \(  V_\bR \) (as a left \( E \)-module) with \( (D^n)^r \) for some positive integer~\( r \).

Again, since all norms on a finite-dimensional real vector space are equivalent, we may assume that the norm on \( V_\bR \cong D^{nr} \) is induced by a norm on \( D \) by letting the norm of an element of \( D^{nr} \) be the maximum of the norms of its coordinates.

Via this identification, write
\[ v_0 = (x_1, \dotsc, x_r), \quad v = (y_1, \dotsc, y_r) \]
where \( x_1, \dotsc, x_r, y_1, \dotsc, y_r \in D^n \).

Reordering \( x_1, \dotsc, x_r \) (and the corresponding \( y_1, \dotsc, y_r \)), we may assume that \( \{ x_1, \dotsc, x_s \} \) forms a maximal right \( D \)-linearly independent subset of \( \{ x_1, \dotsc, x_r \} \) for a suitable positive integer~\( s \leq r \).
Note that \( s \leq n \).
Then there exist \( a_{ij} \in D \) (\( 1 \leq i \leq s < j \leq r \)) such that
\begin{equation} \label{eqn:xj-xi}
x_j = \sum_{i=1}^s x_i a_{ij} \text{ for } s+1 \leq j \leq r.
\end{equation}

By hypothesis, there exists \( e' \in E^\times \) such that \( v = e'v_0 \) or in other words \( y_i = e'x_i \) for all \( i \).
Consequently
\begin{equation} \label{eqn:xj-yj}
y_j = e'x_j = \sum_{i=1}^s e'x_ia_{ij} = \sum_{i=1}^s y_i a_{ij} \text{ for } s+1 \leq j \leq r.
\end{equation}

Since the set \( \{ x_1, \dotsc, x_s \} \) is right \( D \)-linearly independent, it can be extended to form a right \( D \)-basis of \( D^n \).
Let \( h \in \rM_n(D) \) denote the matrix formed using such a basis as its columns.
Then \( h \) is invertible and \( hb_i = x_i \) for \( 1 \leq i \leq s \), where \( \{ b_1, \dotsc, b_n \} \) denotes the standard \( D \)-basis of \( D^n \).
Note that the choices made in constructing \( h \) can be made depending only on \( v_0 \), not on~\( v \).

Note that, since $e'\in E^\times = \GL_n(D)$, the set \( \{y_1, \dotsc, y_s\} \) is also right $D$-linearly independent. Hence, it can be extended to a right $D$-basis of $D^n$ and we can assume, by scaling if necessary, that the norms of the additional vectors are at most $\abs{y_1} + \dotsb + \abs{y_s}$.

Let \( f \in \rM_n(D) \) be the invertible matrix that has this basis as its columns, the first $s$ equal to the \( y_1, \dotsc, y_s \).
Then
\[ fh^{-1}x_i = fb_i = y_i \text{ for } 1 \leq i \leq s. \]
Using \eqref{eqn:xj-xi} and \eqref{eqn:xj-yj}, we deduce that also
\[ fh^{-1}x_j = y_j \text{ for } s+1 \leq j \leq r. \]
In other words, \( e = fh^{-1} \in E^\times \) satisfies \( ev_0 = v \).

By construction, 
\[ \length{f} \leq \abs{y_1} + \dotsb + \abs{y_s}+(n-s) (\abs{y_1} + \dotsb + \abs{y_s}) \leq n\abs{v}. \]
Since \( h \) is independent of \( v \), the proof is complete.
\end{proof}

\subsection{Quantitative fundamental sets for arithmetic groups}

The proof of \cref{fund-set-bound} follows the proof of \cite[Thm.~6.5]{BHC62}.
All we have to do is use the quantitative information from \cref{siegel-intersection} in place of the finiteness statement \cite[Lemma~5.4]{BHC62}.
There are also some minor additional technical steps due to the need to keep track of the finite set \( C \subset \gG(\bQ) \) such that \( C\fS \) is a fundamental set in the ambient group \( \gG(\bR) \) -- this was not needed in \cite[Thm.~6.5]{BHC62} because there \( \gG = \gGL_n \) and so \( C = \{ 1 \} \).

\begin{proof}[Proof of \cref{fund-set-bound}]
Suppose we are given \( u \in \gG(\bR) \)  and \( v_u \in \Aut_{\rho(\gG)}(\Lambda_\bR)v_0 \) such that \( \gH_u = u\gH_{0,\bR}u^{-1} \) is defined over \( \bQ \) and \( \rho(u)v_u \in \Lambda \).
Let \( v_u = \tau(v_0) \) where \( \tau \in \Aut_{\rho(\gG)}(\Lambda_\bR) \), and let \( v = \rho(u)v_u \).

Thanks to \cite[Cor.~6.3]{BHC62}, we may enlarge the lattice \( \Lambda \subset \Lambda_\bQ \) so that it is \( \rho(\Gamma) \)-stable.
For each \( c \in C \), \( c^{-1} \Lambda \) is a lattice in \( \Lambda_\bQ \).
Hence we can choose a lattice \( \Lambda' \subset \Lambda_\bQ \) such that \( c^{-1}\Lambda \subset \Lambda' \) for all \( c \in C \).

By \cref{siegel-intersection}, every \( w \in \rho(\fS)v_u \cap \Lambda' \) has length polynomially bounded with respect to \( \abs{v_u} \).
In particular, for each \( c \in C \),
the set
\[ \rho(\fS)v_u \cap \rho(c^{-1}\Gamma)v \subset \rho(\fS)v_u \cap \Lambda' \]
is finite, so we can choose a finite set \( \{ b_{c,1}, \dotsc, b_{c,m_c} \} \subset \Gamma \) such that
\[ \rho(\fS)v_u \cap \rho(c^{-1}\Gamma)v = \{ \rho(c^{-1}b_{c,1}^{-1})v, \dotsc, \rho(c^{-1}b_{c,m_c}^{-1})v \}. \]
Let \( B_u = \bigcup_{c \in C} \{ b_{c,1}, \dotsc, b_{c,m_c} \} \), which is a finite subset of \( \Gamma \).

By \cref{siegel-intersection}, we have
\[ \abs{\rho(c^{-1}b_{c,i}^{-1})v} \leq \newC* \abs{v_u}^{\newC*} \]
for all \( c \in C \) and \( i \leq m_c \).
Since \( c \) comes from a fixed finite set, we deduce that
\[ \abs{\rho(b_{c,i}^{-1})v} \leq \newC* \abs{v_u}^{\newC*}. \]
This is the length bound on \( \rho(b^{-1}u) v_u \) for \( b \in B_u \) which is required by the statement of the theorem.

Let \( \Gamma_u = \Gamma \cap \gH_u(\bR) \) and \( \cF_{\gH_u} = B_u C \fS u^{-1} \cap \gH_u(\bR) \).
It remains to show that \( \cF_{\gH_u} \) is a fundamental set for \( \Gamma_u \) in \( \gH_u(\bR) \).

Let \( h \in \gH_u(\bR) \subset \gG(\bR) \).
By hypothesis, \( C\fS \) is a fundamental set for \( \Gamma \) in \( \gG(\bR) \) so we can write
\[ hu = \gamma cs \]
where \( \gamma \in \Gamma \), \( c \in C \) and \( s \in \fS \).
Since \( h \in \gH_u(\bR) = \Stab_{\gG(\bR),\rho}(\rho(u)v_0) \), we obtain
\[ \rho(\gamma cs)v_0 = \rho(hu)v_0 = \rho(u)v_0. \]
Applying \( \tau \) we get
\[ \rho(\gamma cs)v_u = v \]
or in other words
\[ \rho(s)v_u = \rho(c^{-1}\gamma^{-1})v \in \rho(\fS)v_u \cap \rho(c^{-1}\Gamma)v. \]
Hence there exists \( b_{c,i} \in B_u \) such that
\[ \rho(c^{-1}b_{c,i}^{-1})v = \rho(s)v_u = \rho(c^{-1}\gamma^{-1})v. \]
In particular, \( \gamma b_{c,i}^{-1} \in \Stab_{\gG(\bR),\rho}(v) \), and we also have \( \gamma b_{c,i}^{-1} \in \Gamma \).
Since \( \tau \in \Aut_{\rho(\gG)}(\Lambda_\bR) \), we have \( \Stab_{\gG(\bR),\rho}(v) = \gH_u(\bR) \).
Thus \( \gamma b_{c,i}^{-1} \in \Gamma_u \) and
\[ h = \gamma b_{c,i}^{-1}.b_{c,i}csu^{-1} \in \Gamma_u \, \cF_{\gH_u}. \]

Thus the \( \Gamma_u \)-translates of \( \cF_{\gH_u} \) cover \( \gH_u(\bR) \).
The fact that there are only finitely many \( \gamma \in \Gamma_u \) for which \( \gamma\cF_{\gH_u} \cap \cF_{\gH_u} \neq \emptyset \) follows from the Siegel property for \( \fS \) (and indeed this implies that \( \cF_{\gH_u} \) also satisfies the Siegel property).
Thus \( \cF_{\gH_u} \) is a fundamental set for \( \Gamma_u \) in \( \gH_u(\bR) \).
\end{proof}

\section{Quantitative reduction theory for quaternion algebras}\label{sec:effective-algebras}

In order to apply \cref{fund-set-bound}, it is necessary to choose a representation \( \rho \) and a vector~\( v_0 \) having the properties described in the theorem.
In this section, we will explain how to construct a suitable representation for our application to unlikely intersections with $E^2$ and quaternionic curves.
This illustrates a method for constructing representations which will be useful for applying \cref{fund-set-bound} to other problems of unlikely intersections in the future, while avoiding many technical complications which occur in more general situations.

Borel and Harish-Chandra's reduction theory considered only a fixed reductive subgroup \( \gH_0 \subset \gG \) (and not its conjugates \( u\gH_0u^{-1} \)). As such, \cite[Thm.~3.8]{BHC62} constructs a representation satisfying the properties (i) and~(ii) of \cref{fund-set-bound} but does not construct the vectors $v_u$.  
Another construction of representations satisfying (i) is given by \cite[Prop.~3.1]{Del82} (based on \cite[Exp.~10, Prop.~5]{Che58}), and Deligne's construction can easily be modified to yield the vectors $v_u$.

However, it is not enough to know just that the vectors \( v_u \) exist.
\Cref{fund-set-bound} gives bounds in terms of \( \abs{v_u} \) so, in order to apply these, we need to control the length \( \abs{v_u} \) in terms of some more intrinsic quantity attached to the subgroup \( u\gH_0u^{-1} \).
For example, in our application the subgroups \( u\gH_0u^{-1} \) will be associated with quaternion algebras and we will bound \( \abs{v_u} \) in terms of the discriminants of (orders in) these algebras.

\subsection{The set-up: quaternionic subgroups of \texorpdfstring{\( \gGSp_4 \)}{GSp4}}

Let \( \gG = \gGSp_4 \), the algebraic group whose \( \bQ \)-points are the invertible linear transformations of~\( \bQ^4 \) which multiply the standard symplectic form by a scalar.
For the standard symplectic form, we use \( \psi \colon \bQ^4 \times \bQ^4 \to \bQ \) represented by the matrix
\[ \fullmatrix{J}{0}{0}{J}, \text{ where } J = \fullmatrix{0}{1}{-1}{0}. \]
The subgroup \( \gH_0 \) is equal to \( \gGL_2 \), embedded block diagonally in \( \gGSp_4 \):
\begin{equation} \label{eqn:sl2-embedding}
\gH_0 = \Bigl\{ \fullmatrix{A}{0}{0}{A} \in \gGSp_4 : A \in \gGL_2 \Big\}.
\end{equation}
If a \( \gG(\bR) \)-conjugate \( u \gH_{0,\bR} u^{-1} \) is defined over \( \bQ \), then its \( \bQ \)-points form the multiplicative group of a (perhaps split) indefinite quaternion algebra over \( \bQ \). We shall prove the following proposition.

\begin{proposition} \label{quaternion-sp4-rep}
Let \( \gG = \gGSp_{4,\bQ} \) and let \( \gH_0 = \gGL_{2,\bQ} \), embedded in \( \gG \) as in~\eqref{eqn:sl2-embedding}.
Let \( \Gamma = \gSp_4(\bZ) \).
Let \( L = \bZ^4 \) and let \( \gG \) act on \( L_\bQ \) in the natural way.

There exist a \( \bQ \)-algebraic representation \( \rho \colon \gG \to \gGL(\Lambda_\bQ) \), where \( \Lambda \) is a finitely generated free \( \bZ \)-module stabilised by \( \Gamma \), a vector \( v_0 \in \Lambda \) and constants \( \newC{quaternion-rep-multiplier} \), \( \newC{quaternion-rep-exponent} \), \( \newC{guh-multiplier} \), \( \newC{guh-exponent} \) such that:
\begin{enumerate}[(i)]
\item \( \Stab_{\gG,\rho}(v_0) = \gH_0 \);
\item the orbit \( \rho(\gG(\bR)) v_0 \) is closed in \( \Lambda_\bR \);
\item for each \( u \in \gG(\bR) \), if the group \( \gH_u = u \gH_{0,\bR} u^{-1} \) is defined over~\( \bQ \), then
\begin{enumerate}
\item there exists \( v_u \in \Aut_{\rho(\gG)}(\Lambda_\bR) v_0 \) such that \( \rho(u) v_u \in \Lambda \) and
\[ \abs{v_u} \leq \refC{quaternion-rep-multiplier} \abs{\disc(R_u)}^{\refC{quaternion-rep-exponent}}; \]
\item there exists \( \gamma \in \Gamma \) and \( h \in \gH_0(\bR) \) such that
\[ \length{\gamma uh} \leq  \refC{guh-multiplier} \abs{\disc(R_u)}^{\refC{guh-exponent}}, \]
\end{enumerate}
where \( R_u \) denotes the order \( \End_{\gH_u}(L)\) of the quaternion algebra \( \End_{\gH_u}(L_\bQ)\).
\end{enumerate}
\end{proposition}

The condition \( v_u \in \Aut_{\rho(\gG)}(\Lambda_\bR) v_0 \) in \cref{quaternion-sp4-rep}(iii)(a) ensures that the element \( \rho(u)v_u \in \Lambda \) satisfies  \( \Stab_{\gG,\rho}(\rho(u)v_u) = \gH_u \).
\cref{quaternion-sp4-rep}(iii)(a) is the bound we need to apply \cref{fund-set-bound}.
\Cref{quaternion-sp4-rep}(iii)(b) is not required for our application to unlikely intersections, but may be useful in its own right -- we can replace \( u \) by \( \gamma uh \) if we replace \( \gH_u \) by \( \gamma \gH_u \gamma^{-1} \), a subgroup of \( \gG \) which gives rise to the same special subvariety of \( \cA_2 \) as \( \gH_u \).


The proof of \cref{quaternion-sp4-rep} will proceed in three steps:
first we construct \( \rho \) and \( v_0 \) satisfying property~(i), then we show that the representation we have constructed possesses property (ii) and then~(iii).
The proofs of properties (ii) and~(iii) are independent of each other, while (iii)(b) is a by-product of the proof of (iii)(a).

\subsection{Construction of representation of \texorpdfstring{\( \gGSp_4 \)}{GSp4}}
We construct the representation \( \rho \) of \( \gGSp_4 \) and the vector \( v_0 \), and define notation which we shall use throughout the rest of the section.

Let \( W = \rM_4(\bQ) \), considered as a \( \bQ \)-vector space.
Define two representations \( \sigma_L, \sigma_R \colon \gG = \gGSp_4 \to \gGL(W) \) by multiplication on the left and on the right:
\[ \sigma_L(g)w = gw, \quad \sigma_R(g)w = wg^{-1}. \]
(The inverse in the formula for \( \sigma_R \) is so that \( \sigma_R \) is a left representation of \( \gG \).)

Let \( E_0 = \rM_2(\bQ) \) and define \( \iota_0 \colon E_0 \to \rM_4(\bQ) \) by
\[ \iota_0(A) = \fullmatrix{A}{0}{0}{A}. \]
Thus \( \gH_0 = \iota_0(\gGL_2) \).
Let \( Z = \iota_0(E_0) \), a \( 4 \)-dimensional \( \bQ \)-linear subspace of \( W \).
Observe that
\[ \Stab_{\gG,\sigma_L}(Z) = \iota_0(E_0) \cap \gG = \gH_0. \]
Similarly, \( \Stab_{\gG,\sigma_R}(Z) = \gH_0 \) but we shall not need this latter fact.

Let \( V = \extpower^4 W \) and let \( \rho_L, \rho_R \colon \gG \to \gGL(V) \) be the representations
\[ \rho_L = \extpower^4 \sigma_L \otimes {\det}^{-1}, \quad \rho_R = \extpower^4 \sigma_R \otimes \det. \]
Then \( \extpower^4 Z \) is a \( 1 \)-dimensional \( \bQ \)-linear subspace of \( V \), with
\[ \Stab_{\gG,\rho_L}(\extpower^4 Z) = \Stab_{\gG,\sigma_L}(Z) = \gH_0. \]
The action of \( \gGL_2(\bQ) \) on \( Z \) via \( \sigma_L \circ \iota_0 \) is the restriction of the left regular representation of \( \rM_2(\bQ) \).
Hence the action of \( \gGL_2(\bQ) \) on \( \extpower^4 Z \) via \( \extpower^4 \sigma_L \circ \iota_0 \) is multiplication by \( (\det_{\gGL_2})^2 = {\det_{\gGL_4}} \circ \iota_0 \).
Therefore the action of \( \gH_0 \) on \( \extpower^4 Z \) via \( \rho_L \) is trivial, so each non-zero vector in \( \extpower^4 Z \) has stabiliser equal to \( \gH_0 \).

Let \( \Lambda = \extpower^4 \rM_4(\bZ) \subset V \).
For later use, we choose a specific element \( v_0 \in (\bigwedge^4 Z) \cap \Lambda \).
Let \( e_1, e_2, e_3, e_4 \) denote the following \( \bZ \)-basis for \( \rM_2(\bZ) \):
\begin{equation} \label{eqn:basis-m2}
e_1 = \fullmatrix{1}{0}{0}{0}, \; e_2 = \fullmatrix{0}{1}{0}{0}, \; e_3 = \fullmatrix{0}{0}{1}{0}, \; e_4 = \fullmatrix{0}{0}{0}{1}
\end{equation}
Then \( \iota_0(e_1), \iota_0(e_2), \iota_0(e_3), \iota_0(e_4) \) form a \( \bZ \)-basis for \( Z \cap \rM_4(\bZ) \), so
\[ v_0 = \iota_0(e_1) \wedge \iota_0(e_2) \wedge \iota_0(e_3) \wedge \iota_0(e_4) \]
is a generator of the rank-\( 1 \) \( \bZ \)-module \( (\extpower^4 Z) \cap \Lambda \).
Then \( \rho_L \) and \( v_0 \) satisfy \cref{quaternion-sp4-rep}(i).

Given \( u \in \gG(\bR) \), we can easily find a vector  \( v_u \in \Aut_{\rho_L(\gG)}(\Lambda_\bR) v_0 \) such that \( \rho_L(u) v_u \in \Lambda \) (that is, the first part of \cref{quaternion-sp4-rep}(iii)(a)).
The vector
\[ \rho_L(u) \rho_R(u) v_0
   = u\iota_0(e_1)u^{-1} \wedge u\iota_0(e_2)u^{-1} \wedge u\iota_0(e_3)u^{-1} \wedge u\iota_0(e_4)u^{-1} \in V_\bR \]
generates the line
\( \extpower^4 u Z_\bR u^{-1} \subset V_\bR \).
If the subgroup \( u\gH_{0,\bR}u^{-1} \subset \gG_\bR \) is defined over \( \bQ \), then so is the linear subspace \( u Z_\bR u^{-1} \subset W_\bR \). (This follows from the fact that $Z$ is the $\QQ$-linear span of $\gH_0(\QQ)$.)
Consequently \( (\extpower^4 u Z_\bR u^{-1}) \cap \Lambda \) is non-empty, so there exists \( d_u \in \bR^\times \) such that
\[ d_u \rho_L(u) \rho_R(u) v_0 \in \Lambda. \]
Now \( d_u \rho_R(u) \in \Aut_{\rho_L}(V_\bR) \) so \( d_u \rho_R(u) v_0 \) has the required property.
This algebraic construction does not control the size of \( d_u \), and hence does not control \( \abs{v_u} \).
Later, in section~\ref{sec:gamma-vu}, we will choose a slightly different \( v_u \) (making use of \( \gamma \) as in \cref{quaternion-sp4-rep}(iii)(b)) allowing us to bound \( \abs{v_u} \).

\medskip

To place this representation in a more general context, we compare it with \cite[Exp.~10, Prop.~5]{Che58}.
Let \( \gG \) be an arbitrary affine \( \bQ \)-algebraic group and \( \gH_0 \subset \gG \) an algebraic subgroup.
Chevalley considers the ring of regular functions \( \bQ[\gG] \), on which \( \gG \) acts by right translations.
The stabiliser of the ideal \( I(\gH_0) \) is equal to \( \gH_0 \).
Choose a finite-dimensional subrepresentation \( W \subset \bQ[\gG] \) which contains a generating set for \( I(\gH_0) \).
Then \( \gH_0 \) is also the stabiliser of \( Z = W \cap I(\gH_0) \).
Let \( d = \dim_\bQ(Z) \).
Then \( \extpower^d W \) is a representation of \( \gG \) in which the line \( \extpower^d Z \) is defined over \( \bQ \) and has stabiliser equal to \( \gH_0 \).
If \( \gH_0 \) is semisimple, it has no non-trivial characters so each non-zero vector in \( \extpower^d Z \) also has stabiliser equal to \( \gH_0 \).
\cite[Prop.~3.1]{Del82} describes how this construction can be modified to obtain a vector $v_0$ (not just a line) with stabiliser equal to \( \gH_0 \) whenever \( \gH_0 \) is reductive.

If we choose \( W \) to be stable under left as well as right translations (denoting the representations by \( \rho_L \) and \( \rho_R \) respectively), then the same argument as in the special case above shows that the line \( \bR^\times \rho_L(u) \rho_R(u) v_0 \) is defined over \( \bQ \) whenever \( u\gH_{0,\bR} u^{-1} \) is defined over \( \bQ \), and so this line contains non-zero rational vectors.

Comparing this general construction with our special case of \( \gG = \gGSp_4 \), \( \gH_0 = \iota_0(\gGL_2) \), we note that in the special case \( I(\gH_0) \) is generated by linear functions on~\( \rM_4 \).
Thus following Chevalley's method, we could choose \( W \) to be the linear dual of \( \rM_4(\bQ) \).
In fact, we chose \( W \) to be \( \rM_4(\bQ) \) itself, and \( Z \) to be the linear subspace of \( \rM_4(\bQ) \) which is annihilated by \( I(\gH_0) \cap \rM_4(\bQ)^\vee \).
The choice of \( \rM_4(\bQ) \) instead of its dual is a matter of convenience.

The representations constructed by Chevalley's method do not necessarily contain a closed orbit \( \rho_L(\gG(\bR))v_0 \), although this can often be achieved by carefully choosing \( W \subset \bQ[\gG] \) and perhaps making some minor modifications using linear algebra constructions.
On the other hand, finding a suitable \( v_u \) with bounded length requires much more detailed arithmetic information about the groups \( \gH_u \).

\subsection{Closed orbit}

We now show that \cref{quaternion-sp4-rep}(ii) holds, that is, the orbit \( \rho_L(\gG(\bR))v_0 \) is closed in \( V_\bR \).
By \cite[Prop.~2.3]{BHC62}, it suffices to prove that \( \rho_L(\gG(\bC))v_0 \) is closed in \( V_\bC \).

We use the following definitions.
If \( V_\bC \) is a vector space over \( \bC \), we say that a subset of \( V_\bC \) is \defterm{homogeneous} if it is non-empty and stable under multiplication by scalars.
In other words, a subset of \( V_\bC \) is homogeneous if and only if it is the cone over some subset of \( \bP(V_\bC) \).
For a non-negative integer~\( d \), a set-theoretic function between vector spaces \( f \colon V_\bC' \to V_\bC \) is \defterm{homogeneous of degree~\( d \)} if
\[ f(\lambda v) = \lambda^d f(v) \text{ for all } \lambda \in \bC, \, v \in V_\bC'. \]

Homogeneous sets and homogeneous maps are useful because of the following lemma, which is equivalent to the fact that a morphism of projective algebraic varieties maps Zariski closed sets to Zariski closed sets.

\begin{lemma} \label{homog-zariski-closed-morphism}
Let \( V_\bC \), \( V_\bC' \) be vector spaces over \( \bC \) (or any algebraically closed field), let \( X \subset V_\bC' \) be a homogeneous Zariski closed subset and let \( f \colon X \to V_\bC \) be a morphism of algebraic varieties which is homogeneous.
If \( f(x) \neq 0 \) for all \( x \in X \setminus \{ 0 \} \),
then \( f(X) \) is a homogeneous Zariski closed subset of~\( V_\bC \).
\end{lemma}

Let \( U = \bC^4 \).
We define a sequence \( (u_1, u_2, u_3, u_4) \in U^4 \) to be \defterm{quasi-symplectic} if it satisfies the conditions
\begin{gather*}
   \psi(u_1, u_3) = \psi(u_1, u_4) = \psi(u_2, u_3) = \psi(u_2, u_4) = 0,
\\ \psi(u_1, u_2) = \psi(u_3, u_4).
\end{gather*}
If \( (u_1, u_2, u_3, u_4) \) is a quasi-symplectic sequence then either:
\begin{enumerate}
\item \( \psi(u_1, u_2) = \psi(u_3, u_4) \neq 0 \), in which case \( (u_1, u_2, u_3, u_4) \) is a non-zero scalar multiple of a symplectic basis for \( (U, \psi) \); or
\item \( \psi(u_1, u_2) = \psi(u_3, u_4) = 0 \), in which case \( u_1, u_2, u_3, u_4 \) are contained in an isotropic subspace of \( U \) for \( \psi \); in particular they are linearly dependent.
\end{enumerate}
Let
\begin{align*}
   \gQ & = \{ g \in \rM_4(\bC) : \text{the columns of } g \text{ form a quasi-symplectic sequence} \}
\\     & = \{ g \in \rM_4(\bC) : \exists \nu(g) \in \bC \text{ such that } \psi(gx, gy) = \nu(g)\psi(x,y) \}
\end{align*}
The set \( \gQ \) is closed under multiplication, but not all of its elements are invertible so it is not a group.
We have \( \gQ \cap \gGL_4(\bC) = \gGSp_4(\bC) \).

Let \( \sigma_L \) denote the action of \( \rM_4(\bC) \) on \( W_\bC = \End_\bC(U) \) by left multiplication (this extends our earlier definition of \( \sigma_L \) as a representation of \( \gG = \gGSp_4 \)).
Let \( \rho'_L \) denote the induced action \( \extpower^4 \sigma_L \) of \( \rM_4(\bC) \) on \( V_\bC = \extpower^4 W_\bC \) (this is a representation of \( \rM_4(\bC) \) as a multiplicative monoid but not as a \( \bC \)-algebra).
Note that \( \rho_L = \rho'_L \otimes \det^{-1} \), but \( \rho_L(g) \) is only defined for \( g \in \gGL_4(\bC) \) while \( \rho'_L(g) \) is defined for all \( g \in \rM_4(\bC) \).
In particular \( \rho'_L \) is defined on \( \gQ \).

In order to prove that \( \rho_L(\gG(\bC))v_0 \) is closed, we find a homogeneous Zariski closed set \( X \subset \extpower^2 U^2 \) such that \( \rho'_L(\gQ)v_0 \) is the image of \( X \) under a homogeneous morphism of varieties~\( \zeta \).
Hence \( \rho'_L(\gQ)v_0 \) is Zariski closed.
We conclude by showing that \( \rho_L(\gG(\bC))v_0 \) is the intersection of \( \rho'_L(\gQ)v_0 \) with a hyperplane in \( V_\bC \).

\begin{lemma} \label{X-is-closed}
The following homogeneous subset of \( \extpower^2 U^2 \) is Zariski closed:
\[ X = \{ (u_1, u_3) \wedge (u_2, u_4) : (u_1, u_2, u_3, u_4) \text{ is quasi-symplectic} \}. \] 
\end{lemma}

\begin{proof}
Let \( \extpowerdec^2 U^2 \) denote the set of decomposable vectors in \( \extpower^2 U^2 \):
\[ \extpowerdec^2 U^2 = \{ x \wedge y : x, y \in U^2 \}. \]
This is the cone over the Grassmannian \( \Gr(2, U^2) \) (embedded in \( \bP(\extpower^2 U^2) \) via the Plücker embedding), so it is a homogeneous Zariski closed subset of \( \extpower^2 U^2 \).

Define a quadratic form \( q \colon U^2 \to \bC \) by \( q((u, v)) = \psi(u, v) \).
Let
\[ X' = \{ x \wedge y \in \extpower^2 U^2 : q_{|\langle x,y \rangle} = 0 \} \]
where \( \langle x,y \rangle \) denotes the linear subspace of \( U^2 \) spanned by \( x \) and \( y \).
Then \( X' \) is the cone over the orthogonal Grassmannian \( \OGr(2, U^2, q) \subset \bP(\extpower^2 U^2) \), so it is a homogeneous Zariski closed subset of \( \extpowerdec^2 U^2 \).

For an element \( (u_1, u_3) \wedge (u_2, u_4) \in \extpowerdec^2 U^2 \), we have:
\begin{align}
    &(u_1, u_3) \wedge (u_2, u_4) \in X'
\notag
\\ \Leftrightarrow {} & q(\lambda(u_1,u_3) + \mu(u_2,u_4)) = 0 \text{ for all } \lambda, \mu \in \bC
\notag
\\ \Leftrightarrow {} & \psi(\lambda u_1 + \mu u_2, \lambda u_3 + \mu u_4) = 0 \text{ for all } \lambda, \mu \in \bC
\notag
\\ \Leftrightarrow {} & \lambda^2 \psi(u_1, u_3) + \lambda \mu \bigl( \psi(u_2, u_3) + \psi(u_1, u_4) \bigr) + \mu^2 \psi(u_2, u_4) = 0 \text{ for all } \lambda, \mu \in \bC
\notag
\\ \Leftrightarrow {} & \psi(u_1, u_3) = \psi(u_2, u_3) + \psi(u_1, u_4) = \psi(u_2, u_4) = 0.
\label{eqn:X'-description}
\end{align}

It follows that we can define linear maps \( \Psi_{ij} \colon X' \to \bC \) for \( i=1,3 \) and $j=2,4$ by
\[ \Psi_{ij}((u_1, u_3) \wedge (u_2, u_4)) = \psi(u_i, u_j). \]
(When $i=1$ and $j=2$ or $i=3$ and $j=4$, the map $\Psi_{ij}$ is defined on the whole of $\extpower^2U^2$. When $i=1$ and $j=4$ or $i=3$ and $j=2$, we require (\ref{eqn:X'-description}).)
Using \eqref{eqn:X'-description}, we conclude that the set~\( X \) from the statement of the lemma is equal to
\[ X' \cap \ker(\Psi_{14}) \cap \ker(\Psi_{23}) \cap \ker(\Psi_{12} - \Psi_{34}). \]
Thus \( X \) is homogeneous and Zariski closed.
\end{proof}

\begin{lemma} \label{fX-is-closed}
\( \rho'_L(\gQ) v_0 \) is a Zariski closed subset of \( V_\bC \).
\end{lemma}

\begin{proof}
Define two linear maps \( \beta_1, \beta_2 \colon U^2 \to W_\bC = \rM_4(\bC) \) by
\begin{align*}
   \beta_1(u_1, u_2) &= \begin{pmatrix} u_1 & 0 & u_2 & 0 \end{pmatrix},
\\ \beta_2(u_1, u_2) &= \begin{pmatrix} 0 & u_1 & 0 & u_2 \end{pmatrix}.
\end{align*}
This notation means that \( \beta(u_1, u_2) \) is the \( 4 \times 4 \) matrix with columns \( u_1, 0, u_2, 0 \), and similarly for \( \beta_2 \).

If \( z_1, z_2, z_3, z_4 \) denote the standard basis of \( U \), then
\begin{equation} \label{eqn:basis-beta}
\iota_0(e_1) = \beta_1(z_1, z_3), \quad \iota_0(e_2) = \beta_2(z_1, z_3), \quad
\iota_0(e_3) = \beta_1(z_2, z_4), \quad \iota_0(e_4) = \beta_2(z_2, z_4).
\end{equation}
The maps \( \beta_1 \) and \( \beta_2 \) commute with the action of \( \gG \) by left multiplication in the sense that
\[ g\beta_i(u_1, u_2) = \beta_i(gu_1, gu_2) \]
for all \( u_1, u_2 \in U \) and \( g \in \gG(\bC) \).
Consequently
\begin{align}
    \rho'_L(\gQ) v_0 = \{ 
  & g \iota_0(e_1) \wedge g \iota_0(e_2) \wedge g \iota_0(e_3) \wedge g \iota_0(e_4) : g \in \gQ \}
\notag
\\  = \{
  & \beta_1(gz_1, gz_3) \wedge \beta_2(gz_1, gz_3) \wedge \beta_1(gz_2, gz_4) \wedge \beta_2(gz_2, gz_4) : g \in \gQ \}
\notag
\\  = \{
  & \beta_1(u_1, u_3) \wedge \beta_2(u_1, u_3) \wedge \beta_1(u_2, u_4) \wedge \beta_2(u_2, u_4) :
\notag
\\& (u_1, u_2, u_3, u_4) \text{ is quasi-symplectic} \}.
\label{eqn:rho_LQv_0}
\end{align}

Define \( f \colon \extpowerdec^2 U^2 \to V_\bC = \extpower^4 W_\bC \) by
\[ f(x \wedge y) = \beta_1(x) \wedge \beta_2(x) \wedge \beta_1(y) \wedge \beta_2(y). \]
This is well-defined,
homogeneous of degree~\( 2 \) and a morphism of varieties.
Thanks to \eqref{eqn:rho_LQv_0}, we have \( \rho'_L(\gQ)v_0 = f(X) \).

If \( x, y \in U^2 \) are linearly independent, then it is easy to check that \( \beta_1(x), \beta_1(y) \) are linearly independent, and that \( \beta_2(x), \beta_2(y) \) are linearly independent.
Furthermore, \( \im(\beta_1) \cap \im(\beta_2) = \{ 0 \} \).
Hence if \( x, y \in U^2 \) are linearly independent, then \( \beta_1(x), \beta_1(y), \beta_2(x), \beta_2(y) \in W_\bC \) are linearly independent.
In other words, if \( x \wedge y \in \bigl( \extpowerdec^2 U^2 \bigr) \setminus \{ 0 \} \), then \( f(x \wedge y) \neq 0 \).

Hence by \cref{homog-zariski-closed-morphism,X-is-closed}, \( f(X) \) is a Zariski closed subset of~\( V_\bC \).
\end{proof}

\begin{lemma} \label{hyperplane-section}
There exists a linear map \( s \colon V_\bC \to \bC \) such that
\[ \rho_L(\gG(\bC))v_0 = \rho'_L(\gQ) v_0 \cap s^{-1}(1). \]
\end{lemma}

\begin{proof}
We continue to use the functions \( \beta_1 \) and \( \beta_2 \) from the proof of \cref{fX-is-closed}.

Let \( \delta \colon W_\bC \to U \) be the linear map which sends the matrix with columns \( \begin{pmatrix} C_1 & C_2 & C_3 & C_4 \end{pmatrix} \) to the sum \( C_1 + C_4 \).
This map is equivariant with respect to multiplication by \( \rM_4(\bC) \) on the left and the compositions \( \delta \circ \beta_1, \delta \circ \beta_2 \colon U^2 \to U \) are the projections onto the two copies of~\( U \).

Taking the fourth exterior power, \( \delta \) induces a linear map
\[ s \colon V_\bC = \extpower^4 W_\bC \to \extpower^4 U \cong \bC. \]
By \eqref{eqn:basis-beta} and the descriptions of \( \delta \circ \beta_1 \), \( \delta \circ \beta_2 \), we have
\[ s(v_0) = \delta\beta_1(z_1, z_3) \wedge \delta\beta_2(z_1, z_3) \wedge \delta\beta_1(z_2, z_4) \wedge \delta\beta_2(z_2, z_4)
   =- z_1 \wedge z_2 \wedge z_3 \wedge z_4. \]
Since \( z_1, z_2, z_3, z_4 \) form a basis for \( U \), \( s(v_0) \neq 0 \).
Hence we can choose the isomorphism \( \extpower^4 U \cong \bC \) so that \( s(v_0) = 1 \).

The linear map \( \delta \) is \( \rM_4(\bC) \)-equivariant with respect to left multiplication.
Consequently \( s \) is \( \rM_4(\bC) \)-equivariant with respect to \( \rho'_L \) on \( V_\bC \) and multiplication by the determinant on \( \extpower^4 U \).
Twisting by \( \det^{-1} \), we deduce that \( s \) is \( \gGL_4(\bC) \)-equivariant with respect to \( \rho_L \) on \( V_\bC \) and the trivial action on \( \extpower^4 U \).
Thus
\[ s(\rho_L(g)v_0) = s(v_0) = 1 \text{ for all } g \in \gG(\bC). \]
Furthermore, if \( g \in \gG(\bC) \), then \( g = \lambda g' \) for some \( g' \in \gG(\bC) \cap \gSL_4(\bC) \) and \( \lambda \in \bC^\times \).
Then \( \rho_L(g) = \rho'_L(g') \).
Since \( \gG(\bC) \subset \gQ \), we conclude that \( \rho_L(g)v_0 \in \rho'_L(\gQ)v_0 \cap s^{-1}(1) \).

Conversely, if \( v \in \rho'_L(\gQ) v_0 \cap s^{-1}(1) \), then we can write \( v = \rho'_L(g) v_0 \) for some \( g \in \gQ \).
Then \( s(v) = \det(g)s(v_0) = \det(g) \).
So \( s(v) = 1 \) forces \( \det(g) = 1 \).
Thus \( g \in \gQ \cap \gSL_4(\bC) \subset \gG(\bC) \) and \( \rho_L(g)v_0 = \rho'_L(g)v_0 = v \).
\end{proof}

\subsection{Discriminants and orders} \label{ssec:discriminants}

Before the proof of \cref{quaternion-sp4-rep}(iii), we prove some results on discriminants, involutions and orders in semisimple algebras.

Let \( D \) be a semisimple \( \bQ \)-algebra.
We define a symmetric \( \bQ \)-bilinear form \( \phi \colon D \times D \to \bQ \) by
\[ \phi(x,y) = \Tr_{D/\bQ}(xy), \]
where \( \Tr_{D/\bQ} \) is the (non-reduced) trace of the regular representation of \( D \).
This form is non-degenerate by \cite[Ch.~I, Prop.~(1.8)]{KMRT98} (note that \cite{KMRT98} refers to the reduced trace, which is a non-zero multiple of \( \Tr_{D/\bQ} \) on each simple factor). For any order $R$ in \( D \), we define the \defterm{discriminant} of \( R \), denoted \( \disc(R) \), to be $\disc(R,\phi)$.

For any involution~\( \dag \) of \( D \), we define another \( \bQ \)-bilinear form \( \phi^\dag \colon D \times D \to \bQ \) by
\begin{equation} \label{eqn:dag-trace-form}
\phi^\dag(x, y) = \Tr_{D/\bQ}(xy^\dag).
\end{equation}
This form is symmetric by \cite[Ch.~I, Cor.~(2.2) and Cor.~(2.16)]{KMRT98}.

\begin{lemma} \label{disc-involution}
Let \( R \) be an order in a semisimple \( \bQ \)-algebra \( D \) and let \( \dag \) be any involution of \( D \).
Then
\( \disc(R, \phi^\dag) = \pm\disc(R) \).
\end{lemma}

\begin{proof}
This is based on the proof of \cite[Prop.~2.9]{GR14}.

Let \( \{ d_1, \dotsc, d_n \} \) be a \( \bZ \)-basis for \( R \) and let \( A \in \gGL_n(\bQ) \) be the matrix such that
\( d_j^\dag = \sum_{i=1}^n A_{ij} d_i \).
Since \( \dag \) is \( \bQ \)-linear and an involution, we have \( A^2 = I \), hence, \( \det(A) = \pm 1 \).
Now
\( \phi^\dag(d_i, d_j) = \sum_{k=1}^n A_{kj} \phi(d_i, d_k) \)
so
\[ \disc(R, \phi^\dag) = \det(A) \disc(R) = \pm \disc(R).
\qedhere \]
\end{proof}

The following lemma is restricted to quaternion algebras because its proof makes use of the fact that the reduced norm is a quadratic form on a quaternion algebra.

\begin{lemma} \label{quat-alg-minkowski}
There exists an absolute constant \( \newC{quat-alg-minkowski-multiplier} \) with the following property.

Let \( R \) be an order in a quaternion algebra \( D \) over \( \bQ \).
Let \( L \) be a left \( R \)-module such that \( L_\bQ \) is isomorphic to the left regular representation of~\( D \).

Then there exists a left \( R \)-ideal \( I \subset R \) such that \( I \) is isomorphic to \( L \) as a left \( R \)-module and
\[ [R:I] \leq \refC{quat-alg-minkowski-multiplier} \abs{\disc(R)}^{3/2}. \]
\end{lemma}

\begin{proof}
This is a generalisation to quaternion algebras of Minkowski's bound for ideal classes in a number field, and the proof is similar.

Choose an isomorphism of \( D \)-modules \( \eta_1 \colon D \to L_\bQ \) and let \( I_1 = \eta_1^{-1}(L) \).

Since \( D \) is a quaternion algebra, it possesses a canonical involution~\( * \) defined by
\( d^* = \Trd_{D/\bQ}(d) - d \), where we write $\Trd_{D/\bQ}$ for the reduced trace.
The canonical involution has the property that
\begin{equation} \label{eqn:nrd-phi}
\phi^*(d,d) = \Tr_{D/\bQ}(dd^*) = 4\Nrd_{D/\bQ}(d) = \pm4 \Nm_{D/\bQ}(d)^{1/2}.
\end{equation}

By \cite[Lemma~1]{Bla48}, there is an element \( s \in I_1 \) satisfying
\[ 0 < \abs{\phi^*(s,s)} \leq \newC* \abs{\disc(I_1, \phi^*)}^{1/4} \]
(the exponent is \( 1/\rk_\bZ(I_1) \)).
Hence by \eqref{eqn:nrd-phi}, there is a constant \( \newC{s-norm-multiplier} \) such that
\begin{equation} \label{eqn:nm-s-bound}
0 < \abs{\Nm_{D/\bQ}(s)} \leq \refC{s-norm-multiplier} \abs{\disc(I_1, \phi^*)}^{1/2}.
\end{equation}

Since \( \Nm_{D/\bQ}(s) \neq 0 \), \( s \) is invertible in \( D \).
Let \( I_2 = I_1 s^{-1} \subset D \).
Then \( I_2 \) is a left \( R \)-module isomorphic to \( L \).
Since \( 1 = ss^{-1} \in I_2 \), we have \( R \subset I_2 \).

Using \eqref{eqn:nm-s-bound}, we can calculate
\[ \abs{\disc(I_2, \phi^*)}
   = \abs{\Nm_{D/\bQ}(s^{-1})}^2 \, \abs{\disc(I_1, \phi^*)}
   \geq \refC{s-norm-multiplier}^{-2}. \]
Consequently, using \cref{disc-involution},
\[ [I_2:R]^2 = \frac{\abs{\disc(R, \phi^*)}}{\abs{\disc(I_2, \phi^*)}} \leq \refC{s-norm-multiplier}^2 \abs{\disc(R)}. \]

Finally let \( I = [I_2:R] I_2 \).
This is contained in \( R \) and is a left \( R \)-submodule of \( D \), so it is a left \( R \)-ideal.
It satisfies
\[ [R:I] = \frac{[I_2:I]}{[I_2:R]} = \frac{[I_2:R]^4}{[I_2:R]} \leq \refC{s-norm-multiplier}^3 \abs{\disc(R)}^{3/2}.
\qedhere \]
\end{proof}

In the following lemma, we note that \( S = \End_R(L) \) is an order in \( \End_D(L_\bQ) \cong \End_D(D) \cong D^\op \).  Thus \( S \) is an order in a quaternion algebra, so \( \disc(S) \) is defined.

\begin{lemma} \label{disc-S-bound}
There exists an absolute constant \( \newC{disc-S-multiplier} \) with the following property.

Let \( R \) be an order in a quaternion algebra \( D \) over \( \bQ \).
Let \( L \) be a left \( R \)-module such that \( L_\bQ \) is isomorphic to the left regular representation of~\( D \).
Let \( S = \End_R(L) \).
Then
\[ \abs{\disc(S)} \leq \refC{disc-S-multiplier} \abs{\disc(R)}^4. \]
\end{lemma}

\begin{proof}
By \cref{quat-alg-minkowski}, \( L \) is isomorphic to a left \( R \)-ideal \( I \subset R \) which satisfies
\[ [R:I] \leq \refC{quat-alg-minkowski-multiplier} \abs{\disc(R)}^{3/2}. \]
Then \( S = \End_R(L) = \End_R(I) \subset \End_D(D) \), 
where the latter is the ring of endomorphisms of \( D \) as a left \( D \)-module.

We can define a multiplication-reversing function \( \mu \colon D \to \End_D(D) \) by
\[ \mu(d) x = xd \text{ for all } d, x \in D. \]
This is a \( \bQ \)-algebra isomorphism \( D^\op \to \End_D(D) \) \cite[Ch.~8, Lemma~1.10]{Sch85}.

If \( r \in R \) and \( x \in I \), then we have \( \mu(r) x = xr \in R \)
since \( I \subset R \) and \( R \) is closed under multiplication.
Hence
\[ [R:I] \mu(r) x \in I. \]
Thus \( [R:I]\mu(r) \in \End_R(I) \) for all \( r \in R \).

Let \( \phi_S \) denote the trace form on \( \End_D(D) = S_\bQ \).
Since \( \mu \) is an algebra isomorphism, it pulls back \( \phi_S \) to the trace form on \( D^\op \), which is equal to the trace form on \( D \).
Hence \( \disc(\mu(R), \phi_S) = \disc(R) \).

Since \( [R:I]\mu(R) \subset S \), we conclude that
\[ \abs{\disc(S)} \leq \abs{\disc([R:I]\mu(R), \phi_S)} = [R:I]^2\,\abs{\disc(R)} \leq \refC{quat-alg-minkowski-multiplier}^2 \abs{\disc(R)}^4.
\qedhere \]
\end{proof}

\subsection{Choice of \texorpdfstring{\( v_u \)}{vu}, \texorpdfstring{\( \gamma \)}{gamma} and \texorpdfstring{\( h \)}{h}} \label{sec:gamma-vu}

Throughout this section, \( c_n \) will denote absolute constants (in particular, independent of~\( u \)).

We will now prove \cref{quaternion-sp4-rep}(iii).
Thus, we are given \( u \in \gG(\bR) = \gGSp_4(\bR) \) such that the algebraic group \( \gH_u = u\gH_{0,\bR}u^{-1} \subset \gG_\bR \) is defined over~\( \bQ \).
Multiplying \( u \) by a scalar does not change \( \gH_u \), so we may assume that \( u \) multiplies the symplectic form \( \psi \) by \( \pm 1 \); consequently \( \det(u) = 1 \).

Since \( \gH_u \) is defined over \( \bQ \), the \( \bR \)-vector space \( u\iota_0(E_{0,\bR})u^{-1} \) is also defined over~\( \bQ \).
Hence the \( \bQ \)-algebra
\[ E = \rM_4(\bQ) \cap u\iota_0(E_{0,\bR})u^{-1} \]
satisfies \( E_\bR = u\iota_0(E_{0,\bR})u^{-1} \).

Let \( D_0 \) and \( D \) denote the centralisers in \( \rM_4(\bQ) \) of \( \iota_0(E_0) \) and \( E \), respectively.
Then \( D_\bR = uD_{0,\bR} u^{-1} \), so there is an isomorphism of \( \bR \)-algebras \( \alpha \colon D_\bR \to D_{0,\bR} \) defined by
\[ \alpha(d) = u^{-1}du. \]

Note that
\[ D_0 = \Bigl\{ \fullmatrix{aI}{bI}{cI}{dI} \in \rM_4(\bQ) : a, b, c, d \in \bQ \Bigr\}. \]
Hence the ``transpose'' involution of \( \rM_4(\bQ) \) restricts to an involution of \( D_0 \), which we denote by \( t \).
This involution is positive, which is to say that the corresponding trace form (see \eqref{eqn:dag-trace-form}) is positive definite.
Let
\[ \dag = \alpha^{-1} \circ t \circ \alpha \colon D_\bR \to D_\bR. \]

Let \( L = \bZ^4 \).
Let \( R = \End_E(L)\), which is the order in~\( D \) consisting of those elements preserving $L$.

\begin{lemma} \label{dag-R} 
The quadratic form \( \phi^\dag \) takes integer values on \( R \).
\end{lemma}

\begin{proof}
Thanks to our choice of symplectic form \( \psi \) on \( \bQ^4 \), we have \( \psi(d_0x, y) = \psi(x, d_0^t y) \)
for all \( x, y \in \bR^4 \) and \( d_0 \in D_{0,\bR} \).
Using this, the fact that \( u \) multiplies \( \psi \) by \( \pm 1 \), and the definition of~\( \dag \), we can calculate, for $d\in D_\bR$,
\begin{align*}
    \psi(d x, y)
  & = \psi(u \alpha(d) u^{-1} x, \, y)
    = \pm\psi(\alpha(d) u^{-1} x, \, u^{-1} y)
\\& = \pm\psi(u^{-1} x, \, \alpha(d)^t u^{-1} y)
    = \pm\psi(u^{-1} x, \, \alpha(d^\dag) u^{-1} y)
\\& = \pm\psi(x, \, u \alpha(d^\dag) u^{-1} y)
    = \pm\psi(x, d^\dag y).
\end{align*}

Since \( \psi \) is a perfect pairing on \( L \), this implies that \( \dag \) maps \( R \) into \( R \).
It follows that \( \phi^\dag(x,y) = \Tr_{D/\bQ}(xy^\dag) \in \bZ \) for all \( x, y \in R \).
\end{proof}

Observe that $L_\QQ$ is isomorphic to the left regular representation of $D_0$. Hence, $\alpha$ induces an isomorphism between $L_\RR$ and the left regular representation of $D_\RR$ and it follows easily that this isomorphism can be scaled to produce an isomorphism between $L_\QQ$ and the left regular representation of $D$. Therefore, by \cref{quat-alg-minkowski}, there is a left \( R \)-ideal \( I \subset R \) which is isomorphic to \( L \) as a left \( R \)-module, such that
\[ [R:I] \leq \refC{quat-alg-minkowski-multiplier} \abs{\disc(R)}^{3/2}. \]
Choose a left \( R \)-module isomorphism \( \eta \colon I \to L \).

Fix an isomorphism \( \eta_0 \colon D_0 \to \bQ^4 \) of left \( D_0 \)-modules (independent of \( u \)).

\begin{lemma} \label{uh}
There exists \( h \in \gH_0(\bR) \) such that \( \eta \alpha^{-1} \eta_0^{-1} = uh \) in \( \Aut(\bR^4) \).
\end{lemma}

\begin{proof}
We can view \( D_\bR \) as a left \( D_{0,\bR} \)-module with the action given by
\begin{equation} \label{eqn:alpha-action}
d\cdot x = \alpha^{-1}(d)x.
\end{equation}
Now \( \alpha^{-1}\eta_0^{-1} \colon \bR^4 \to D_\bR \) is an isomorphism of left \( D_{0,\bR} \)-modules with respect to the natural action on \( \bR^4 \) and the action~\eqref{eqn:alpha-action} on \( D_\bR \).

Since \( \eta \colon D_\bR \to \bR^4 \) is an isomorphism of \( D_\bR \)-modules, it is also an isomorphism of \( D_{0,\bR} \)-modules with respect to the action \eqref{eqn:alpha-action} on \( D_\bR \) and the natural action conjugated by~\( u \) on \( \bR^4 \).
(We use here the fact that \( \alpha^{-1} \) is conjugation by~\( u \).)

Finally \( u^{-1} \colon \bR^4 \to \bR^4 \) is an isomorphism of \( D_{0,\bR} \)-modules with respect to the natural action conjugated by~\( u \) on the domain and the natural action on the target.

Composing these, we deduce that \( u^{-1} \eta \alpha^{-1} \eta_0^{-1} \) is an automorphism of \( \bR^4 \) with its natural action of \( D_{0,\bR} \).
In other words, \( u^{-1}\eta\alpha^{-1}\eta_0^{-1} \) lies in the centraliser of \( D_{0,\bR} \) in \( \rM_4(\bR) \).
By the double centraliser theorem, this centraliser is equal to \( \iota_0(E_{0,\bR}) \) and so its group of invertible elements is equal to \( \iota_0(\gGL_2(\bR)) = \gH_0(\bR) \).
\end{proof}

\begin{lemma} \label{det-lower-bound}
The absolute value of \( \det(h) \) is uniformly bounded.
\end{lemma}

\begin{proof}
Let \( \phi_0^t \) denote the bilinear form \( \phi_0^t(x, y) = \Tr_{D_0/\bQ}(xy^t) \) on \( D_0 \), and let \( \phi_L \) denote the bilinear form on \( \bQ^4 \) given by \( \phi_L(x, y) = \phi_0^t(\eta^{-1}_0(x), \eta^{-1}_0(y)) \).

Since \( \alpha \) is an isomorphism of \( \bR \)-algebras, it preserves traces, so
\begin{equation} \label{eqn:phidag-phiL}
\phi^\dag(x,y) = \phi_0^t(\alpha(x), \alpha(y)) = \phi_L(\eta_0\alpha(x), \eta_0\alpha(y)).
\end{equation}
Consequently
\[ \disc(\phi^\dag, I) =  \disc(\phi^\dag, \eta^{-1}(L)) = \disc(\phi_L, \eta_0\alpha\eta^{-1}(L))
   = \det(\eta_0\alpha\eta^{-1})^2 \disc(\phi_L, L). \]
Thanks to \cref{uh} and noting that \( \det(u) = 1 \), this can be rewritten as
\[ \det(h)^2 = \det(uh)^2 = \disc(\phi_L, L) / \disc(\phi^\dag, I). \]

By \cref{dag-R}, \( \phi^\dag \) takes integer values on \( R \) and hence on \( I \).
So \( \disc(\phi^\dag, I) \) is a positive integer.
So \( \det(h)^2 \leq \disc(\phi_L, L) \), which is a constant.
\end{proof}

\begin{lemma} \label{imI-basis-bound}
There exists a \( \bZ \)-basis \( \{ d_1', d_2', d_3', d_4' \} \) for \( I \) such that the coordinates of the vectors \( \eta_0\alpha(d_1'), \eta_0\alpha(d_2'), \eta_0\alpha(d_3'), \eta_0\alpha(d_4') \in \bR^4 \) are polynomially bounded in terms of \( \abs{\disc(R)} \).
\end{lemma}

\begin{proof}
By \cref{dag-R}, \( \phi^\dag \) takes integer values on~\( I \).
Furthermore \( \dag \) is a positive involution (because \( t \) is a positive involution on~\( D_0 \)) so \( \phi^\dag \) is a positive definite quadratic form.
Hence by \cite[Thm.~5]{Wey40} there exists a \( \bZ \)-basis \( \{ d_1', d_2', d_3', d_4' \} \) for \( I \) satisfying
\[ \phi^\dag(d_i', d_i') \leq \newC* \disc(I, \phi^\dag) \]
for \( i = 1, 2, 3, 4 \).
Hence by \eqref{eqn:phidag-phiL}, the values \( \phi_L(\eta_0\alpha(d_i'), \eta_0\alpha(d_i')) \) are bounded by a constant multiple of \( \disc(I, \phi^\dag) \).
Since \( \phi_L \) is a fixed positive definite quadratic form on \( \bR^4 \), this implies that the coordinates of the vectors \( \eta_0\alpha(d_i') \) are polynomially bounded in terms of \( \disc(I, \phi^\dag) \).

Finally, by \cref{disc-involution,quat-alg-minkowski}, we have
\[ \disc(I, \phi^\dag) = [R:I]^2 \, \abs{\disc(R)} \leq \newC* \abs{\disc(R)}^4.
\qedhere \]
\end{proof}

Let \( \{ \ell_1, \ell_2, \ell_3, \ell_4 \} \) denote the standard basis for \( L \).
Since \( \{ \eta(d_1'), \eta(d_2'), \eta(d_3'), \eta(d_4') \} \) is a \( \bZ \)-basis for \( L \), there is a matrix \( \gamma' \in \gGL_4(\bZ) \) such that \( \ell_i = \gamma'\eta(d_i') \) for each~\( i \).


\begin{lemma} \label{entries'-bound}
The entries of the matrices \( \gamma' uh, (\gamma' uh)^{-1} \in \gGL_4(\bR) \) are polynomially bounded in terms of \( \abs{\disc(R)} \).
\end{lemma}

\begin{proof}
Let \( A = \gamma' uh = \gamma' \eta \alpha^{-1} \eta_0^{-1} \in \gGL_4(\bR) \).
We have
\[ \ell_i = \gamma' \eta(d_i') = A \eta_0 \alpha(d_i'). \]
By \cref{imI-basis-bound}, the coordinates of the vectors \( A^{-1}\ell_i = \eta_0\alpha(d_i') \) are polynomially bounded, or in other words, the entries of the matrix \( A^{-1} \) are polynomially bounded in terms of \( \abs{\disc(R)} \).

Meanwhile, \( \abs{\det(\gamma')} = \det(u) = 1 \) so \( \abs{\det(A)} = \abs{\det(h)} \).
By \cref{det-lower-bound} we deduce that \( \abs{\det(A^{-1})} \) is bounded below by a positive constant.
Hence by Cramer's rule, the entries of the matrix~\( A \) are also polynomially bounded in terms of \( \abs{\disc(R)} \).
\end{proof}

%
%

The following lemma establishes \cref{quaternion-sp4-rep}(iii)(b).
Note that it is not required for the proof of \cref{quaternion-sp4-rep}(iii)(a): the subsequent arguments proving \cref{quaternion-sp4-rep}(iii)(a) do not use the fact that \( \gamma \in \gSp_4(\bZ) \), so they would still work with \( \gamma' \) instead of \( \gamma \).

\begin{lemma} \label{entries-bound}
There exists \( \gamma \in \Gamma = \gSp_4(\bZ) \) such that the entries of \( \gamma uh \) and \( (\gamma uh)^{-1} \) are polynomially bounded in terms of \( \abs{\disc(R)} \).
\end{lemma}

\begin{proof}
We have \( uh \in \gGSp_4(\bR) \).
Consequently
\[ \psi(\gamma'^{-1} \ell_i, \gamma'^{-1} \ell_j) = \pm\det(uh)^{1/2} \, \psi((uh)^{-1}\gamma'^{-1} \ell_i, \, (uh)^{-1} \gamma'^{-1} \ell_j). \]
Using \cref{det-lower-bound,entries'-bound}, we conclude that the values \( \psi(\gamma'^{-1} \ell_i, \gamma'^{-1} \ell_j) \) are polynomially bounded in terms of \( \abs{\disc(R)} \).

Hence by \cite[Lemma~4.3]{Orr15}, there exists a symplectic \( \bZ \)-basis \( \{ s_1, s_2, s_3, s_4 \} \) for \( (L, \psi) \) whose coordinates with respect to \( \{ \gamma'^{-1}\ell_1, \gamma'^{-1}\ell_2, \gamma'^{-1}\ell_3, \gamma'^{-1}\ell_4 \} \) are polynomially bounded in terms of \( \abs{\disc(R)} \).
Applying \( \gamma' \), we deduce that the coordinates of \( \gamma' s_1, \gamma' s_2, \gamma' s_3, \gamma' s_4 \) with respect to the standard basis are polynomially bounded.

Let \( \gamma \in \gGL_4(\bZ) \) be the matrix such that \( \ell_i = \gamma s_i \) for each~\( i \).
Since \( \{ s_1, s_2, s_3, s_4 \} \) is a symplectic basis, we have \( \gamma \in \Gamma \).
We have just shown that the coordinates of \( \gamma' s_i = \gamma' \gamma^{-1} \ell_i \) are polynomially bounded.
In other words, the entries of the matrix \( \gamma' \gamma^{-1} \) are polynomially bounded in terms of \( \abs{\disc(R)} \).

Multiplying \( (\gamma'uh)^{-1} \) by \( \gamma' \gamma^{-1} \) and applying \cref{entries'-bound}, we deduce that the entries of \( (\gamma uh)^{-1} \) are polynomially bounded in terms of \( \abs{\disc(R)} \).
Thanks to \cref{det-lower-bound}, \( \abs{\det((\gamma uh)^{-1})} \) is bounded below by a positive constant, so it follows that the entries of \( \gamma uh \) are also polynomially bounded in terms of \( \abs{\disc(R)} \).
\end{proof}

Let \( S = E \cap \rM_4(\bZ) = \End_R(L) \) and \( S_0 = \iota_0(E_0) \cap \rM_4(\bZ) \).
Set \[ d_u = (\disc(S) / \disc(S_0))^{1/2}\text{ and }v_u = d_u \rho_R(\gamma u) v_0 \in V_\bR. \]
We shall use this \( v_u \) to prove \cref{quaternion-sp4-rep}(iii)(a).
Note first that \( d_u \rho_R(\gamma u) \in \Aut_{\rho_L}(\Lambda_\bR) \).

\begin{lemma} \label{vu-Lambda}
\( \rho_L(u) v_u \in \Lambda \).
\end{lemma}

\begin{proof}
Since \( \rho_R(\gamma) \) stabilises \( \Lambda \), in order to show that \( \rho_L(u)v_u \in \Lambda \), it suffices to show that
\[ d_u \rho_R(u) \rho_L(u) v_0 \in \Lambda. \]

Since \( v_0 \) is a generator for the rank-\( 1 \) \( \bZ \)-module \( (\extpower^4 Z) \cap \Lambda = \extpower^4 S_0 \),
it follows that \( \rho_R(u) \rho_L(u) v_0 \) is a generator for \( \extpower^4 uS_0u^{-1} \subset \extpower^4 E_\bR \).

Consider a matrix \( B \in \gGL(E_\bR) \) such that \( B(uS_0u^{-1}) = S \).
Conjugation by \( u \) maps the trace form \( \phi_{E_0} \) on \( \iota_0(E_{0,\bR}) \) to the trace form \( \phi_E \) on \( E_\bR \), so we have
\[ \disc(uS_0u^{-1}, \phi_E) = \disc(S_0, \phi_{E_0}) = \disc(S_0). \]
Consequently \( \det(B)^2 = \disc(S) / \disc(S_0) \).
In other words, \( \det(B) = \pm d_u \).

It follows that
\[ \extpower^4 S = \det(B) \extpower^4 uS_0u^{-1} = \pm d_u \extpower^4 uS_0u^{-1} \]
and so \( d_u \rho_R(u) \rho_L(u) v_0 \) is a generator for \( \extpower^4 S \subset \Lambda \).
\end{proof}

\newpage

\begin{lemma} \label{length-gammavu}
\( \abs{v_u} \leq \newC* \abs{\disc(R)}^{\newC*} \).
\end{lemma}

\begin{proof}
The action of \( \gH_0 \) on the line \( \extpower^4 Z \) via \( \rho_R \) is trivial, for the same reasons as the action via \( \rho_L \) is trivial.
Therefore
\begin{align*}
    v_u
  & = d_u \rho_R(\gamma u) v_0
\\& = d_u \rho_R(\gamma u h) v_0
\\& = d_u \iota_0(e_1) \, (\gamma uh)^{-1} \wedge \dotsb \wedge \iota_0(e_4) \, (\gamma uh)^{-1}.
\end{align*}

By \cref{disc-S-bound}, \( d_u \) is polynomially bounded in terms of \( \abs{\disc(R)} \).
By \cref{entries-bound}, the entries of \( (\gamma uh)^{-1} \) are polynomially bounded in terms of \( \abs{\disc(R)} \).
We conclude that the coordinates, and hence the length, of \( v_u \) are polynomially bounded in terms of \( \abs{\disc(R)} \).
\end{proof}

\Cref{vu-Lambda,length-gammavu} complete the proof of \cref{quaternion-sp4-rep}(iii)(a).

\section{Unlikely intersections in \texorpdfstring{$\mathcal{A}_2$}{A2}}\label{sec:unlikely-int-proofs}

In this section, we prove Theorems \ref{ZPSh} and \ref{ZPShUn}. We first need some preliminary material.

\subsection{Realising \texorpdfstring{$\mathcal{A}_2$}{A2} as a Shimura variety}
Recall that $\cA_2$ denotes the (coarse) moduli space of principally polarised abelian surfaces.
To realise $\cA_2$ as a Shimura variety, we let $(\gG,X)$ denote the Shimura datum for which $\gG=\gGSp_4$ and $X$ is isomorphic to $\cH_2 \cup \cH_2^-$, where $\cH_2$ and $\cH_2^-$ are respectively the Siegel upper and lower half-spaces of genus $2$. (Recall that $X$ is a $\gG(\RR)$-conjugacy class of morphisms $\mathbb{S}\to\gG_\RR$, where $\mathbb{S}=\Res_{\CC/\RR}\GG_{m,\CC}$. We will henceforth identify $X$ with $\cH_2 \cup \cH_2^-$.) We let $K=\gG(\hat{\ZZ})$, where $\hat{\ZZ}=\prod_p\ZZ_p$, the product ranging over all finite primes $p$. Then $K$ is a compact open subgroup of $\gG(\AAA_f)$, where $\AAA_f$ denotes the finite rational adeles, and $\mathcal{A}_2$ is equal to the Shimura variety whose complex points are given by
\[\Sh_K(\gG,X)=\gG(\QQ)\backslash X\times\gG(\AAA_f)/K.\]
As is easily seen, this is isomorphic to the quotient $\gSp_4(\ZZ)\backslash \mathcal{H}_2$.

\subsection{Quaternionic curves and \texorpdfstring{$E^2$}{E\^{}2} curves}\label{ShE2}
Quaternionic curves and $E^2$ curves are the images in $\cA_2$ of Shimura varieties of PEL type, via maps induced by morphisms of Shimura data.
We recall the construction of Shimura varieties of PEL type attached to quaternion algebras over~$\bQ$, following \cite[sec.~8]{Mil05}.

Let $B$ denote a quaternion algebra over~$\QQ$ such that $B\otimes_\QQ\RR$ is isomorphic to $\gM_2(\RR)$ and let $\dag$ be a positive involution of $B$.
(Positive involutions exist for any such $B$, as explained in \cite[pp.~195--196]{Mum74}.)
As explained in \cite[p.~196]{Mum74}, we can choose the isomorphism $B \otimes_\QQ \RR \to \gM_2(\RR)$ in such a way that $\dag$ corresponds to transpose of matrices, so $(B \otimes_\QQ \CC, \dag)$ has type~C in the sense of \cite[Prop.~8.3]{Mil05}.

Choose \( \alpha \in B \) such that \( \alpha = -\alpha^\dag \).
Define a symplectic form on $B$ by the formula
\[ \psi_\alpha(x, y) = \Tr_{B/\bQ}(x \alpha y^\dag). \]
If $\alpha \in B^\times=B\setminus\{0\}$,  then $\psi_\alpha$ is non-degenerate and $(B, \psi_\alpha)$ (with $B$ acting via the left regular representation) is a symplectic $(B, \dag)$-module as in \cite[sec.~8]{Mil05}.

Let $\gH_B$ denote the centraliser of $B^\times$ (acting on $B$ via the left regular representation) in $\gGL(B)$, seen as a $\QQ$-algebraic group.
Since every symplectic form $\psi:B \times B \to \QQ$ which satisfies $\psi(bx, y) = \psi(x, b^\dag y)$ has the form $\psi_\alpha$ for some $\alpha$ such that $\alpha = -\alpha^\dag$, and (for a given positive involution~$\dag$) the set of such \( \alpha \)s forms a one-dimensional \( \bQ \)-linear subspace of \( B \),
every element of $\gH_B$ preserves $\psi_\alpha$ up to multiplication by a scalar. Therefore, $\gH_B$ is the group of $\bQ$-linear automorphisms of $B$ commuting with the $B$-action and preserving the symplectic form $\psi_\alpha$ up to similitudes. In other words, $\gH_B$ is equal to the group denoted $G$ in \cite[sec.~8]{Mil05}.
(Note: if $B^\op$ denotes the opposite algebra of $B$, we have $\gH_B(\QQ) \cong (B^\op)^\times \cong B^\times$ where the second isomorphism uses the fact that $B$ is a quaternion algebra.)


By \cite[Prop.~8.14]{Mil05}, there is a unique Shimura datum \( (\gH_B, X_B) \) such that each \( h \in X_B \) defines a complex structure on $B \otimes_\bQ \bR$ for which the symmetric form \( \psi(x, h(i)y) \) is positive or negative definite.
As a Hermitian symmetric domain, $X_B$ is isomorphic to the union of the upper and lower half-planes in $\CC$.
 
Choosing a symplectic $\QQ$-basis for $(B, \psi_\alpha)$, the tautological action of $\gH_B$ on $B$ gives rise to an injective group homomorphism $\gH_B \to \gG$.
Thanks to the properties of $X_B$ given to us by \cite[Prop.~8.14]{Mil05}, this induces an embedding of Shimura data $(\gH_B, X_B) \to (\gG, X)$.
Letting $K_B=\gH_B(\AAA_f)\cap K$, we obtain a morphism
\begin{align*}
\Sh_{K_B}(\gH_B,X_B)\rightarrow\Sh_K(\gG,X)
\end{align*}
of algebraic varieties.
The irreducible components of the images of such morphisms are, by definition, special curves in $\mathcal{A}_2$. If $B$ is isomorphic (over $\QQ$) to $\gM_2(\QQ)$, we obtain $E^2$ curves, and otherwise we obtain quaternionic curves. Any such curve parametrises abelian surfaces with multiplication by an order in $B$.

The Shimura data $(\gG, X)$ and $(\gH_B,X_B)$ all have reflex field $\QQ$. Therefore, $\Sh_{K_B}(\gH_B,X_B)$, $\Sh_K(\gG,X)$ and $\Sh_{K_B}(\gH_B,X_B)\rightarrow\Sh_K(\gG,X)$ are all defined over $\QQ$, but $\Sh_{K_B}(\gH_B,X_B)$ often has geometrically irreducible components which are not defined over $\QQ$.
Hence, the action of $\Aut(\CC/\QQ)$ on $\cA_2(\CC)$ preserves the image of $\Sh_{K_B}(\gH_B,X_B)\rightarrow\Sh_K(\gG,X)$ but permutes its irreducible components and so acts on the set of quaternionic curves and on the set of $E^2$ curves in $\cA_2$. 
From the theory of complex multiplication of abelian varieties, we know that $\Aut(\CC/\QQ)$ acts on the set of special points in $\cA_2$.
Consequently, $\Aut(\CC/\QQ)$ acts on
\[ \Sigma_{\Quat} = \bigcup_{Z\in\mathcal{S}} Z \setminus Z^{\rm sp}, \]
where $\mathcal{S}$ denotes the set of quaternionic curves in $\cA_2$ and $Z^{\rm sp}$ denotes the set of the special points contained in $Z$. Similarly, $\Aut(\CC/\QQ)$ acts on \( \Sigma_{E^2} \).

Another way to obtain these families of special subvarieties is as follows.
Let $B_0 = \rM_2(\QQ)$.
Let $B_0$ act on $\QQ^4$ via the left regular representation, with respect to the basis given by \eqref{eqn:basis-m2} (which is a symplectic basis with respect to the form \( \psi_\alpha \) where \( \alpha = \fullsmallmatrix{0}{-1}{1}{0} \)).
Let $\gH_0 \subset \gG$ be the centraliser of this action of $B_0$ in $\gG$.
Then $\gH_0$ is equal to the image of $\gGL_2$ embedded block diagonally, as in \eqref{eqn:sl2-embedding}.
Let
\begin{gather*}
X_0 = \Bigl\{ \fullmatrix{\tau}{0}{0}{\tau} \in \cH_2 : {\rm Im}(\tau)>0 \Bigr\},
\\ X_0^{\pm} = \Bigl\{ \fullmatrix{\tau}{0}{0}{\tau} \in \cH_2 : {\rm Im}(\tau) \neq 0 \Bigr\}.
\end{gather*}
Then \( (\gH_0, X_0^{\pm}) \) is the unique Shimura subdatum of \( (\gG, X) \) with underlying group $\gH_0$, and $X_0$ is the only connected component of $X_0^{\pm}$ contained in $\cH_2$. We obtain a morphism of Shimura varieties \( \cA_1 \to \cA_2 \)
(where \( \cA_1 \) denotes the moduli space of elliptic curves), which, in terms of moduli, sends an elliptic curve \( E \) with its principal polarisation \( \lambda \) to the principally polarised abelian surface \( (E \times E, \lambda \times \lambda) \).

For any point $x_0\in X_0$, we have $X_0=\gH_0^\der(\RR) x_0$ (recall that $\gH_0^\der(\RR)=\gSL_2(\RR)$ is connected) and its image in $\mathcal{A}_2$ is an $E^2$ curve. For any $g\in\gG(\RR)$ such that $\gH=g\gH_{0,\RR}g^{-1}$ is defined over $\QQ$ the image of $gX_0$ in $\mathcal{A}_2$ is a special curve, and $\gH$ is isomorphic (as a $\QQ$-group) to $\gH_B$ for some quaternion algebra $B$ as above. If $\gH$ is isomorphic to $\gGL_{2,\QQ}$, then we obtain an $E^2$ curve, and, otherwise, we obtain a quaternionic curve.

\begin{lemma}\label{arise}
Every quaternionic or $E^2$ curve in $\mathcal{A}_2$ is the image of $gX_0$ for some $g \in \gG(\RR)$ such that $g\gH_{0,\RR}g^{-1}$ is defined over $\QQ$.
\end{lemma}
 
\begin{proof}
Let $Z$ be a quaternionic or $E^2$ curve in $\cA_2$.
Let $B$ be the generic endomorphism algebra of the abelian surfaces parametrised by $Z$, and let $\dag$ be the Rosati involution of $B$.
Choose an analytic irreducible component $Y$ of the preimage of~$Z$ in~$X$.

The inclusion $\gG \to \gGL_4$ induces a variation $\mathcal{V}$ of $\QQ$-Hodge structures on $X$ with trivial underlying local system $X \times \QQ^4$.
The restriction $\mathcal{V}_{|Y}$ has endomorphism algebra $B$ and its generic Mumford--Tate group $\gH \subset \gG$ is the centraliser of $B$ in $\gG$.
Thus $\gH$ is the image of one of the homomorphisms $\gH_B \to \gG$ defined above.

The choice of basis \eqref{eqn:basis-m2} induces an isomorphism of $\QQ$-vector spaces $\QQ^4 \to B_0$.
Choose an isomorphism of $\RR$-algebras with involutions $(B_0 \otimes_\QQ \RR, t) \to (B \otimes_\QQ \RR, \dag)$.
The action of $B$ on $\mathcal{V}_{|Y}$ gives rise to a $B$-module structure on $\bQ^4$, and this is isomorphic to the left regular representation of $B$ on itself.
Thus we get an isomorphism of $B$-modules $B \to \QQ^4$.
Composing these isomorphisms (after extending scalars to $\RR$), we get an isomorphism of $\RR$-vector spaces
\[ \RR^4 \to B_0 \otimes_\QQ \RR \to B \otimes_\QQ \RR \to \RR^4 \]
or in other words an element $g \in \gGL_4(\RR)$.
Via the isomorphism $\RR^4 \to B_0 \otimes_\QQ \RR$, the standard symplectic form on $\RR^4$ satisfies $\psi(bx, y) = \psi(x, b^t y)$ for all $b \in B_0$, and the isomorphism $B \otimes_\QQ \RR \to \RR^4$ behaves similarly with respect to $(B, \dag)$.
Since the spaces of symplectic forms satisfying these conditions are one-dimensional, the composed isomorphism maps the standard symplectic form to a multiple of itself; in other words, $g \in \gG(\RR)$.

Comparing the actions of $B \otimes_\QQ \RR$ and $B_0 \otimes_\QQ \RR$, we see that $\gH_\RR = g\gH_{0,\RR}g^{-1}$.
It follows that $g^{-1}Y$ is a connected component of a Shimura subdatum of $(\gG, X)$ with underlying group $\gH_0$.
The only such Shimura subdatum is $(\gH_0, X_0^{\pm})$, so $g^{-1}Y$ is a connected component of $X_0^{\pm}$.

If $g^{-1}Y \neq X_0$, replace $g$ by $g\diag(1,-1,1,-1)$.
Since $\diag(1,-1,1,-1) \in \gH_0(\RR)$ and it swaps the two connected components of $X_0^{\pm}$, after this replacement we will still have $\gH_\RR = g\gH_{0,\RR}g^{-1}$ but now $g^{-1}Y = X_0$.

Thus we get $Y = gX_0$ and $Z$ is the image of $Y$ in $\cA_2$.
\end{proof}

\subsection{Complexity}

As in \cite{DR}, we will need to define a notion of complexity. That is, to each $E^2$ or quaternionic curve $Z$ in $\mathcal{A}_2$, we attach a natural number $\Delta(Z)$, which we refer to as the complexity of $Z$. The complexity is defined in terms of the generic endomorphism algebra of abelian surfaces parametrised by $Z$.

Let $g\in\gG(\RR)$ be such that $\gH=g\gH_{0,\RR}g^{-1}$ is defined over $\QQ$. Then the image $Z$ of $gX_0$ in $\mathcal{A}_2$ is an $E^2$ or quaternionic curve and, by Lemma \ref{arise} every $E^2$ or quaternionic curve is obtained this way. We define the complexity $\Delta(Z)$ of $Z$ to be $\abs{\disc(R)}$, where $R$ denotes the ring $\End_{\gH}(\ZZ^4)$ of $\ZZ$-linear endomorphisms of $\ZZ^4\subset\QQ^4$ commuting with $\gH(\QQ)\subset\gG(\QQ)\subset\rM_4(\QQ)$.
Note that this ring $R$ is the generic endomorphism ring of the abelian surfaces parameterised by $Z$.
Indeed, for every non-special point of $Z$, the associated abelian surface (over $\CC$) has endomorphism ring isomorphic to $R$.

We are now in a position to state the Galois orbits conjecture which appears in \cref{ZPSh}.

\begin{conjecture}\label{GOSh}
Let $\Sigma$ denote $\Sigma_{\Quat}$ or $\Sigma_{E^2}$ and let $C\subset\mathcal{A}_2$ denote an irreducible Hodge generic algebraic curve. Let $L$ be a finitely generated subfield of $\CC$ over which $C$ is defined.

There exist positive constants $\newC{GOSh-mult}$ and $\newC{GOSh-exp}$ such that, for any point $s\in C\cap\Sigma$, if we let $Z$ denote the (unique) special curve containing $s$, then
\[\#\Aut(\CC/L)\cdot s\geq\refC{GOSh-mult}\Delta(Z)^{\refC{GOSh-exp}}.\]
\end{conjecture}

\subsection{The fixed data}

We write $\Gamma$ for the subgroup $\gSp_4(\ZZ) \subset \gG(\QQ)$.
We let $\pi:\mathcal{H}_2\to\mathcal{A}_2$ denote the (transcendental) uniformisation map. We choose a Siegel set $\mathfrak{S}\subset\gG(\RR)^+$ (associated with the standard Siegel triple) such that, for some finite set $C\subset\gG(\QQ)$, $\mathcal{F}_\gG=C\mathfrak{S}$ is a fundamental set for $\Gamma$ in $\gG(\RR)^+$. We write $\mathcal{F}$ for $\mathcal{F}_\gG x_0$, where $x_0\in\mathcal{H}_2$ is the point whose stabiliser in $\gG(\RR)^+$ is the maximal compact subgroup appearing in the definition of $\mathcal{F}_\gG$.

By Proposition \ref{quaternion-sp4-rep}, we can fix a finitely generated, free $\ZZ$-module $\Lambda$, a representation $\rho:\gG\rightarrow\gGL(\Lambda_\QQ)$ such that $\Lambda$ is stabilised by $\rho(\Gamma)$, an element $v_0\in\Lambda$, and positive constants $\newC{eff-alg-mult}$ and $\newC{eff-alg-exp}$ such that
\begin{enumerate}[(i)]
\item \( \Stab_{\gG,\rho}(v_0) = \gH_0 \);
\item the orbit \( \rho(\gG(\bR)) v_0 \) is closed in \( \Lambda_\bR \);
\item for each \( u \in \gG(\bR) \), if the group \( \gH_u = u \gH_{0,\bR} u^{-1} \) is defined over~\( \bQ \), then there exists \( v_u \in \Aut_{\rho(\gG)}(\Lambda_\bR) v_0 \) such that \( \rho(u) v_u \in \Lambda \) and
\[ \abs{v_u} \leq \refC{eff-alg-mult} \abs{\disc(R_u)}^{\refC{eff-alg-exp}}, \]
where \( R_u \) denotes the order \( \End_{\gH_u}(\ZZ^4)\) of the quaternion algebra \( \End_{\gH_u}(\QQ^4)\). 
\end{enumerate}

By Theorem \ref{fund-set-bound}, we can then fix positive constants \( \newC{eff-red-mult} \) and \( \newC{eff-red-exp} \) with the following property:
for every \( u \in \gG(\bR) \), if \( \gH_u = u \gH_{0,\bR} u^{-1} \) is defined over~\( \bQ \), then there exists a fundamental set for \( \Gamma \cap \gH_u(\bR) \) in \( \gH_u(\bR) \) of the form
\[ B_u C \fS u^{-1} \cap \gH_u(\bR), \]
where \( B_u \subset \Gamma \) is a finite set such that
\[ \abs{\rho(b^{-1}u) v_u}  \leq  \refC{eff-red-mult} \abs{v_u}^{\refC{eff-red-exp}} \]
for every \( b\in B_u \).

Choosing a basis, we obtain $\Lambda=\ZZ^d$ and we may refer to the height $\rH(v)$ of any $v\in\Lambda$ (defined as the maximum of the absolute values of the coordinates). For any $v \in \Lambda_\RR$, we write $\gG(v) = \Stab_{\gG_\RR,\rho}(v)$.



\begin{proposition}\label{param}
Let $P\in\Sigma_{\Quat}\cup\Sigma_{E^2}$. Then there exists $z\in\pi^{-1}(P)\cap\mathcal{F}$ and $v\in\Aut_{\rho(\gG)}(\Lambda_\RR)\rho(\gG(\RR))v_0\cap\Lambda$ such that $z(\mathbb{S})\subset\gG(v)$ and 
\begin{align*}
\rH(v)\leq\refC{eff-red-mult} \refC{eff-alg-mult}^{\refC{eff-red-exp}}\abs{\disc(R)}^{\refC{eff-alg-exp}\refC{eff-red-exp}},
\end{align*}
where $R$ denotes the ring  \( \End_{\gG(v)}(L) \subset \rM_4(\bZ) \).
\end{proposition}

\begin{proof}
Let $z\in\pi^{-1}(P)\cap\mathcal{F}$ and let $Y$ denote the smallest pre-special subvariety of $\mathcal{H}_2$ containing $z$. Then $\pi(Y)$ is an $E^2$ or quaternionic curve and so
\[Y=g\gH_0^\der(\RR)x_0=g\gH_0^\der(\RR)g^{-1}\cdot g x_0=\gH^\der(\RR)\cdot g x_0,\]
where $g\in\gG(\RR)$, $\gH$ is a $\QQ$-subgroup of $\gG$ isomorphic to $\gH_B$ for some quaternion algebra $B$, as above, and $\gH_{\RR}=g\gH_{0,\RR}g^{-1}$.
Since $\gH$ is the Mumford-Tate group of $Y$ (that is, the smallest $\QQ$-subgroup of $\gG$ containing $x(\bS)$ for all $x\in Y$), we have $z(\bS) \subset \gH(\RR)$.
By \cref{quaternion-sp4-rep}, we obtain $v\in\Aut_{\rho(\gG)}(\Lambda_\RR)v_0$ such that $\rho(g)v\in\Lambda$ and
\[|v|\leq \refC{eff-alg-mult}\abs{\disc(R)}^{\refC{eff-alg-exp}},\]
where $R=\End_{\gH}(L)$.
Note that
\[\gG(\rho(g)v) = g \gG(v) g^{-1} = g \gG(v_0) g^{-1} = \gH_\RR. \]

By \cref{fund-set-bound} (with $u=g$), we obtain a finite set $B_g\subset\Gamma$ such that
\[\mathcal{F}_{\gH}=B_g C\mathfrak{S}g^{-1}\cap\gH(\RR)\]
is a fundamental set for $\Gamma_\gH=\Gamma\cap\gH(\RR)$ in $\gH(\RR)$, and
\[|\rho(b^{-1}g)v|\leq \refC{eff-red-mult}|v|^{\refC{eff-red-exp}}\]
for all $b\in B_g$. 
In particular, since $z \in \gH(\RR)gx_0$, we can write
\[z=\gamma b s x_0\]
for some $\gamma\in\Gamma_\gH$, $b\in B_g$, and $s\in C\mathfrak{S}$. Hence,
\[z':=b^{-1}\gamma^{-1} z\in\mathcal{F}\cap\pi^{-1}(P).\]
Furthermore, we have
\[z'(\mathbb{S})\subset \gG(\rho(b^{-1}\gamma^{-1}g) v)=\gG(\rho(b^{-1} g)v)=b^{-1}\gH_\RR b,\]
where we use the fact that $\gamma\in\gH(\RR)=\gG(\rho(g)v)(\RR)$. Finally, from the above, we have
\[\abs{\rho(b^{-1} g)v}
\leq \refC{eff-red-mult}\abs{v}^{\refC{eff-red-exp}}
\leq \refC{eff-red-mult}( \refC{eff-alg-mult}\abs{\disc(R)}^{\refC{eff-alg-exp}})^{\refC{eff-red-exp}}=\refC{eff-red-mult} \refC{eff-alg-mult}^{\refC{eff-red-exp}}\abs{\disc(R)}^{\refC{eff-alg-exp}\refC{eff-red-exp}}\]
and
\[R=\End_{\gH}(L)=\End_{b^{-1}\gH b}(b^{-1}L)=\End_{b^{-1}\gH b}(L)=\End_{\gG(\rho(b^{-1}g)v)}(L).\]
Therefore, since $\rho(b^{-1}g)v\in\Lambda$, we conclude that $z'$ and $\rho(b^{-1}g)v$ satisfy the conditions of the proposition. 
\end{proof}

\begin{corollary}\label{suitable}
Let $b\in\RR$. The set of $E^2$ or quaternionic curves $Z$ satisfying $\Delta(Z)\leq b$ is finite.
\end{corollary}

\begin{proof}
Let $Z$ be an $E^2$ or quaternionic curve satisfying $\Delta(Z)\leq b$ and let $P\in Z$ be a Hodge generic point on $Z$. Therefore, $P\in\Sigma_{\Quat}\cup\Sigma_{E^2}$ and, applying \cref{param} and the first paragraph of its proof, we obtain $v\in\Lambda$ satisfying
\begin{align*}
\rH(v)\leq\refC{eff-red-mult} \refC{eff-alg-mult}^{\refC{eff-red-exp}}b^{\refC{eff-alg-exp}\refC{eff-red-exp}}
\end{align*}
such that $Z$ is the image in $\cA_2$ of an orbit of $\gG(v)^\der(\RR)$. As in the proof of \cref{arise}, there is only one Shimura subdatum of $(\gG,X)$ associated with $\gG(v)$ and so the result follows.
\end{proof}

\subsection{Proof of Theorem \ref{ZPSh} for quaternionic curves}
Let $\mathcal{C}=\pi^{-1}(C)\cap\mathcal{F}$ -- a set definable in the o-minimal structure $\RR_{\rm an,exp}$ (see \cite{kuy:ax-lindemann} for more details). Let $L$ be a finitely generated field of definition for $C$. 
Let $P\in C\cap\Sigma_{\Quat}$. Varying over $\sigma\in\Aut(\CC/L)$, we obtain points $\sigma(P)\in C\cap\Sigma_{\Quat}$ and, for each $\sigma$, we let $z_\sigma\in\mathcal{F}\cap\pi^{-1}(\sigma(P))$ and we let $v_\sigma\in\Aut_{\rho(\gG)}(\Lambda_\RR)\rho(\gG(\RR))v_0\cap\Lambda$ be the elements afforded to us by Proposition \ref{param}. That is, $z_\sigma(\mathbb{S})\subset \gG(v_\sigma)$ and 
\begin{align*}
\rH(v_\sigma)\leq\refC{eff-red-mult} \refC{eff-alg-mult}^{\refC{eff-red-exp}}\abs{\disc(R_\sigma)}^{\refC{eff-alg-exp}\refC{eff-red-exp}},
\end{align*}
where $R_\sigma$ denotes the ring  \( \End_{\gG(v_\sigma)}(L) \subset \rM_4(\bZ) \). As above, $\abs{\disc(R_\sigma)}=\Delta(\sigma(Z))$. Note that we also have $z_\sigma\in\mathcal{C}$.

We obtain a set $\Theta$ of tuples $(v_\sigma,z_\sigma)\in\Lambda\times\mathcal{C}$ belonging to the definable set
\[D=\{(v,z)\in \Lambda_\RR\times\mathcal{C}: v\in \Aut_{\rho(\gG)}(\Lambda_\RR)\rho(\gG(\RR))v_0,\ z(\mathbb{S})\subset \gG(v)\}.\]
Let $\pi_1:D\rightarrow\Lambda_\RR$ and let $\pi_2:D\rightarrow\mathcal{C}$ denote the projection maps. By \cref{GOSh}, we have
\begin{align*}
A:=\#\pi_2(\Theta) = \#\Aut(\CC/L)\cdot P & = \#\Aut(\CC/L)\cdot\sigma(P)
\\& \geq \refC{GOSh-mult} \abs{\disc(R_\sigma)}^{\refC{GOSh-exp}} \geq \newC* H(v_\sigma)^{\refC{GOSh-exp}/\refC{eff-alg-exp}\refC{eff-red-exp}}.
\end{align*}
Applying \cite[Theorem 9.1]{DR} (a variant of \cite[Corollary 7.2]{HP:atyp}), in the case $l=0$, $k=1$, $T=(1/\newC*A)^{\refC{eff-alg-exp}\refC{eff-red-exp}/\refC{GOSh-exp}}$, and $\eps < \refC{GOSh-exp}/\refC{eff-alg-exp}\refC{eff-red-exp}$, we conclude that either
\begin{enumerate}
\item $A=\#\pi_2(\Theta)$ is bounded, hence $\Delta(Z)$ is bounded and the theorem holds, or
\item there exists a continuous definable function
\[\beta:[0,1]\rightarrow D\]
such that $\beta_1=\pi_1\circ\beta$ is semi-algebraic, $\beta_2=\pi_2\circ\beta$ is non-constant, $\beta(0)\in \Theta$, and $\beta_{|(0,1)}$ is real analytic.
\end{enumerate}

Therefore, it suffices to rule out the latter possibility.
To that end, suppose that we have such a function. By definable choice, there exists a semi-algebraic function
\[\tilde{\beta}_1:[0,1]\rightarrow  \Aut_\rho(\Lambda_\RR)\rho(\gG(\RR))\]
such that $\tilde{\beta}_1(t)\cdot v_0=\beta_1(t)$ for all $t\in [0,1]$. We let $v_t=\beta_1(t)$, $g_t=\tilde{\beta}_1(t)$, and $z_t=\beta_2(t)$. Since $z_t(\mathbb{S})\subset \gG(v_t)$ and $g_t\in \Aut_{\rho(\gG)}(\Lambda_\RR)\rho(\gG(\RR))$, we have 
\[(g^{-1}_tz_t)(\mathbb{S})\subset g^{-1}_t \gG(v_t)g_t=\gG(g^{-1}_tv_t)=\gG(v_0) = \gH_0(\RR).\]
We conclude that $g^{-1}_tz_t$ lies on the unique pre-special subvariety of $\mathcal{H}_2$ associated with $\gH_0^\der(\RR)$, namely, $X_0$.



On the one hand, there exists $0<t_1\leq 1$ such that $\beta_2([0,t_1])$ is contained in a single irreducible analytic component $\tilde{C}$ of $\pi^{-1}(C)$. By \cite[Theorem 1.3]{uy:algebraic-flows} (the inverse Ax--Lindemann conjecture), $\mathcal{H}_2$ is the smallest algebraic subset of $\mathcal{H}_2$ containing $\tilde{C}$.

Let $B\subset\Lambda_\CC$ denote the Zariski closure of $\beta_1([0,t_1])\subset\Lambda_\RR$. By definable choice, there exists a complex algebraic set $\tilde{B}\subset  \Aut_\rho(\Lambda_\CC)\rho(\gG(\CC))$ of dimension at most~$1$ whose image under the algebraic map $g\mapsto g\cdot v_0$ is $B$. Using the superscript $^\vee$ to denote the compact dual of a hermitian symmetric domain, we obtain a complex algebraic set $\tilde{B}\times X_0^\vee$ of dimension at most $2$. Note that $\tilde{B}\cdot \mathcal{H}_2^\vee=\mathcal{H}_2^\vee$. Hence, $\tilde{B}\cdot X_0^\vee\subset\mathcal{H}^\vee_2$ is algebraic of dimension at most $2$.

On the other hand, $\tilde{C}$ is an irreducible complex analytic curve having an uncountable intersection with $\tilde{B}\cdot X_0^\vee$ (in particular, it includes $\beta_2([0,t_1])$ because $g^{-1}_tz_t\in X_0$ and $z_t\in \tilde{C}$ for all $t\in[0,t_1]$). Therefore, $\tilde{C}$ is contained in $\tilde{B}\cdot X_0^\vee$, hence, so is $\mathcal{H}^\vee_2$. However, $\dim\mathcal{H}^\vee_2=3$, and we arrive at a contradiction.

\subsection{Proof of Theorem \ref{ZPSh} for \texorpdfstring{$E^2$}{E\^{}2} curves}

The proof is the same as in the previous section, working with $\Sigma_{E^2}$ instead of $\Sigma_{\Quat}$.

\subsection{Proof of Theorem \ref{ZPShUn}}

If $C$ is an algebraic curve over a number field and $\fA\to C$ is an abelian scheme of relative dimension $2$, we say that $s\in C(\Qbar)$ is a quaternionic point if the endomorphism algebra of the fiber $\fA_s$ is a quaternion algebra over $\QQ$ not isomorphic to $\rM_2(\QQ)$.

We claim that it suffices to prove the following theorem.

\begin{theorem}\label{scheme}
Let $C$ be an irreducible algebraic curve and let $\mathfrak{A}\rightarrow C$ be a principally polarised non-isotrivial abelian scheme of relative dimension $2$ such that $\End(\fA_{\ov\eta}) = \bZ$, where $\ov\eta$ denotes a geometric generic point of $C$. 

Suppose that $C$ and $\mathfrak{A}$ are defined over a number field $L$ and that there exist a curve \( C' \), a semiabelian scheme \( \fA' \to C' \) and an open immersion \( \iota \colon C \to C' \), all defined over \( \Qbar \), such that \( \fA \cong \iota^* \fA' \) and there is a point \( s_0 \in C'(\ov\bQ) \setminus C(\ov\bQ) \) for which the fibre \( \fA'_{s_0} \) is a torus.

Then there exist positive constants $\newC{GOE2-scheme-mult}$ and $\newC{GOE2-scheme-exp}$ such that, for any quaternionic point $s\in C$,
\[\#\Aut(\CC/L)\cdot s\geq\refC{GOE2-scheme-mult}\abs{\disc(\End(\fA_s))}^{\refC{GOE2-scheme-exp}}.\]
\end{theorem}

To see that \cref{scheme} implies \cref{ZPShUn}, consider $C$ as in Theorem \ref{ZPShUn}. Then $C$ is defined over a number field $L$ and, furthermore, we can construct a curve $\tilde{C}'$, a finite surjective morphism $q:\tilde{C}\rightarrow C$, $s_0\in\tilde{C}'$, and a semiabelian scheme $\mathfrak{A}'\to\tilde{C}'$ as in  \cite[Proposition 9.4]{DO19}. We can find a finite extension $\tilde{L}/L$ such that $\tilde{C}'$, $q:\tilde{C}\rightarrow C$, $s_0$ and $\mathfrak{A}'$ are all defined over $\tilde{L}$. The abelian scheme \( \fA'_{|\tilde{C}} \to \tilde{C} \) and the point \( s_0 \in \tilde{C}'(\ov\bQ) \) satisfy the conditions of Theorem \ref{scheme} and so, for any quaternionic point $\tilde{s}\in \tilde{C}$,
\[\#\Aut(\CC/\tilde{L})\cdot \tilde{s}\geq\refC{GOE2-scheme-mult}\abs{\disc(\End(\fA'_{\tilde{s}}))}^{\refC{GOE2-scheme-exp}}.\]
If $s\in C\cap\Sigma_{\Quat}$, then we can find a quaternionic point $\tilde{s}\in\tilde{C}$ such that $q(\tilde{s})=s$, and since $q$ is finite, 
\[\#\Aut(\CC/L)\cdot s\geq\newC{GOE2-scheme2-mult}\#\Aut(\CC/\tilde{L})\cdot \tilde{s}.\]
Let $Z$ denote the unique special curve in $\cA_2$ containing $s$.
Since $s$ is a Hodge generic point of $Z$, the endomorphism ring of the associated abelian surface $A_s$ is isomorphic to the generic endomorphism ring of $Z$ and so $\Delta(Z) = \abs{\disc(\End(A_s))}$.
Since also $A_s = \fA'_{\tilde{s}}$, we can combine the above inequalities to obtain
\[ \#\Aut(\CC/L)\cdot s\geq\refC{GOE2-scheme-mult}\refC{GOE2-scheme2-mult}\Delta(Z)^{\refC{GOE2-scheme-exp}}, \]
that is, \cref{GOSh}.

Therefore, it remains to prove Theorem \ref{scheme}. 

\begin{proof}[Proof of Theorem \ref{scheme}]
After a finite extension, we may assume that  \( C' \),  \( \fA' \to C' \), \( \iota \colon C \to C' \), and \( s_0 \) are all defined over $L$. Since \( \End(\fA_{\ov\eta}) = \bZ \) and \( \dim(\fA_{\ov\eta}) = 2 \), the Mumford--Tate group of \( \fA_{\ov\eta} \) is \( \gGSp_{4,\bQ} \) (see \cite[Section 2.F]{DO19}).
Thus \( \fA \to C \) satisfies the conditions of \cite[Theorem~8.1]{DO19}, as modified in \cite[Remark~8.6]{DO19}.

Let $s\in C$ be a quaternionic point. The image of $s$ under the map $C\rightarrow\mathcal{A}_2$ induced by $\fA\to C$ is in the intersection between the image of $C$ and a quaternionic curve. We deduce that $s\in C(\Qbar)$.

Now $\End(\fA_s)\otimes\QQ$ is a non-split quaternion algebra, so cannot inject into $\rM_2(\QQ)$. Hence, $A_s$ is exceptional in the sense of \cite[Section~8]{DO19}. Therefore, by \cite[Theorem~8.1]{DO19}, $h(s)$ is polynomially bounded in terms of $[L(s):L]$, where $h$ denotes a Weil height on $C'$. Let $h_F$ denote the stable Faltings height. As proved in \cite[p. 356]{Fal83},
\[ \abs{h_F(\fA_s) -  h(s)}  = \rO(\log  h(s)). \]
We conclude that \( h_F(\fA_s) \) is polynomially bounded in terms of \( [L(s):L] \). 


In order to deduce a bound for $\disc(\End(A_s))$, we use the following theorem of Masser and Wüstholz.

\begin{theorem} \cite[p.~641]{MW:endo}
Given positive integers $n$, $d$ and $\delta$, there are constants $\newC{mw-ends-multiplier} = \refC{mw-ends-multiplier}(n, d, \delta)$ and $\newC{mw-ends-exponent} = \refC{mw-ends-exponent}(n)$, with the following property.
Let $A$ be an abelian variety of dimension $n$ defined over a number field $k$ of degree $d$, equipped with a polarisation of degree~$\delta$.
Let $\dag$ be the Rosati involution of $\End(A)$ associated with this polarisation and let $\phi^\dag$ be the bilinear form on $\End(A)$ defined by \eqref{eqn:dag-trace-form}.
Then $\disc(\End_k(A), \phi^\dag)$ is at most $\refC{mw-ends-multiplier} \max(1, h_F(A))^{\refC{mw-ends-exponent}}$.
\end{theorem}

As remarked in \cite{MW:endo} immediately following the statement of this theorem, one can replace $\End_k(A)$ by $\End_{\bC}(A)$ (the endomorphism ring which appears in the statement of \cref{scheme}) because one can find a finite extension $K/k$ of degree bounded only in terms of $n$ such that $\End_K(A) = \End_\bC(A)$.
Furthermore, as stated near the bottom of \cite[p.~650]{MW:endo}, the constant $\refC{mw-ends-multiplier}(n, d, \delta)$ is polynomial in $d$ and $\delta$, for a polynomial which depends only on $n$.
Using also \cref{disc-involution} to see that $\abs{\disc(\End(A), \phi^\dag)} = \abs{\disc(\End(A))}$, we conclude that there are constants $\newC{mw-ends2-multiplier}$, $\newC{mw-ends2-exponent}$ depending only on $n$ such that, for all $A$ as in the theorem, we have
\begin{equation} \label{eqn:mw-ends2}
\abs{\disc(\End(A))} \leq \refC{mw-ends2-multiplier} \max(\delta, d, h_F(A))^{\refC{mw-ends2-exponent}}.
\end{equation}

In our case, we have always $n=2$ and $\delta=1$.
So applying \eqref{eqn:mw-ends2} together with the fact that $h_F(A_s)$ is polynomially bounded in terms of $[L(s):L]$ completes the proof of \cref{scheme}.
\end{proof}

\subsection*{Acknowledgements}

Both authors are grateful to the anonymous referees for their suggestions which have improved the paper.
They would also like to thank the Universities of Reading, Oxford and Warwick.

\subsection*{Funding}

This work was supported by the Engineering and Physical Sciences Research Council [EP/S029613/1 to C.D., EP/T010134/1 to M.O.].

\providecommand{\noopsort}[1]{}
\providecommand{\bysame}{\leavevmode\hbox to3em{\hrulefill}\thinspace}
\providecommand{\MR}{\relax\ifhmode\unskip\space\fi MR }
\providecommand{\MRhref}[2]{%
  \href{http://www.ams.org/mathscinet-getitem?mr=#1}{#2}
}
\providecommand{\href}[2]{#2}

\end{document}